\documentclass[12pt, reqno, a4paper]{amsart}
\usepackage{amsmath,mathrsfs}
\usepackage{amssymb}
\usepackage{color}
\usepackage{calligra}
\usepackage{dsfont}
\usepackage{pstricks,pst-node}
\usepackage[latin1]{inputenc}
\usepackage{todonotes}
\usepackage{hyperref}

\newcommand\NoBlackBoxes{\global\overfullrule0pt}
\NoBlackBoxes

\usepackage{enumitem}
\setitemize{noitemsep,topsep=0pt,parsep=0pt,partopsep=0pt}

\newcommand{\eps}{\varepsilon}

\newcommand{\N}{\mathbb{N}}

\renewcommand{\P}{\mathbb{P}}

\makeatletter
\let\serieslogo@\relax
\let\@setcopyright\relax

\allowdisplaybreaks
\theoremstyle{plain}
\newtheorem{theorem}{Theorem}[section]
\newtheorem{lemma}[theorem]{Lemma}
\newtheorem{corollary}[theorem]{Corollary}
\newtheorem{proposition}[theorem]{Proposition}

\theoremstyle{definition}

\theoremstyle{remark}
\newtheorem{remark}[theorem]{Remark}

\renewcommand{\P}{{\mathbb{P}}}

\newcommand{\R}{{\mathbb{R}}}

\renewcommand{\epsilon}{\varepsilon}
\renewcommand{\phi}{\varphi}

\newcommand{\be}{\begin{equation}}
	\newcommand{\ee}{\end{equation}}

\numberwithin{equation}{section}

\usepackage[latin1]{inputenc}

\title{On the Effect of Bottlenecks in Block Spin Models}

\author[Isabel Lammers]{Isabel Lammers}
\address[Isabel Lammers]{Fachbereich Mathematik und Informatik,
	University of M\"unster,
	Einsteinstra\ss e 62,
	48149 M\"unster,
	Germany}

\email[Isabel Lammers]{isabel.lammers@uni-muenster.de}

\author[Matthias L\"owe]{Matthias L\"owe}
\address[Matthias L\"owe]{Fachbereich Mathematik und Informatik,
	University of M\"unster,
	Einsteinstra\ss e 62,
	48149 M\"unster,
	Germany}

\email[Matthias L\"owe]{maloewe@uni-muenster.de}


\begin{document}
	\begin{abstract}
		We study a bottleneck spin model with $N$ spins, split into two Curie--Weiss models at low temperature with a bottleneck between them. We propose multiple ways of how to realize such a bottleneck and study its influence on the phase transition in the thermodynamic limit $N \to \infty$.
        
        In all versions of this model we prove the existence of a threshold that determines whether or not the presence of the bottleneck is felt in the phase transition. This threshold depends on the size of the bottleneck and the interaction strength through it.
	\end{abstract}

	\maketitle
	\section{Introduction}
	
	This note is motivated by the question of how the geometry, i.e.\ the graph structure, of a graph influences the behavior of a spin system defined on it. Closely related questions have been studied extensively for random walks on graphs. For example, Diaconis asked in \cite{Dicaonis_learn}, Section~5, Question~4, how one can add a perfect matching to a regular graph so as to optimize the mixing time of the random walk on the resulting graph. In \cite{HLS22}, the authors show that adding a random perfect matching to a sequence of graphs $(G_n)$ with diverging sizes and uniformly bounded degrees (under mild connectivity assumptions) leads, with high probability, to cutoff for the random walk on the augmented graphs. This result was extended in \cite{Baran2024} to a setting in which the perfect matchings are assigned an $n$-dependent weight. As long as this weight is sufficiently large, the resulting random walk again exhibits cutoff with high probability, whereas for weights that are too small, cutoff occurs only if the original random walk already exhibits this phenomenon. On the other hand, random walks on dense Erd\H{o}s--R\'enyi random graphs appear to display a universal behavior. For hitting times, this can be seen, for instance, in \cite{LoTe23} and \cite{LoTe25}.
	
	Similarly, spin models -- and in particular the Ising model on dense Erd\H{o}s--R\'enyi random graphs -- exhibit universal behavior at the level of laws of large numbers and fluctuations for the magnetization; see, for example, \cite{BG93} and \cite{KaLS19a, KaLS19b}. For sparse graphs, however, the situation is considerably more delicate, as demonstrated in, among others, \cite{Dembo_Montanari_2010a, Dembo_Montanari_2010b, van_der_Hofstad_et_al_2014, van_der_Hofstad_et_al_2015, van_der_Hofstad_et_al_2015b}.
	
	In this note, we also study the Ising model on a graph. Our work is inspired by the block spin Ising model introduced in \cite{GC08} and later again in \cite{BRS19}. The block spin Ising model is an Ising model on a fully connected graph that is partitioned into two or more groups. Unlike the Curie--Weiss model (the mean-field Ising model), the interaction strength within a group differs from the interaction strength between groups. From this perspective, the models proposed here can be viewed as an attempt to introduce a bottleneck between groups and to investigate how this bottleneck affects the behavior of the Ising model. In light of the results in \cite{HLS22} and \cite{Baran2024}, our work can also be interpreted as introducing a perfect matching -- random or deterministic, which is immaterial for our purposes -- between two otherwise disconnected Curie--Weiss models and examining whether this perturbation influences the phase transition, more precisely, the number of equilibrium points in the low temperature phase.
	
	The remainder of this note is organized as follows. In Section~2, we introduce several variations of models that one can interpret as a bottleneck between two Curie-Weiss models. We start in Section~2.1 with the basic model in its simplest form, namely two
	otherwise unconnected Curie--Weiss models coupled by a weighted perfect matching and identify a critical interaction strength at which the matching affects the low-temperature limiting states. In Section~2.2 we introduce the natural variation of the aforementioned model where each matching edge will be kept with a certain probability $p(N)$ and otherwise deleted. The third model is introduced in Section~2.3 and instead of a perfect matching lets the two Curie-Weiss models interact through a third much smaller Curie-Weiss model. In Sections~3 through 6, we prove the existence of this phase transition, which qualitatively depend on the choice of the specific model.

	\section{Models and main results}

    \subsection{A perfect matching}\label{sec:PerfectMatch}
	We begin with the simplest version of the model we have in mind. It is a block spin model with two blocks of equal size (see, e.g., \cite{BRS19}), coupled by a very sparse interaction between the blocks. The spins are labeled $1,\ldots,N$, where $N$ is assumed to be even. Spins with labels $1,\ldots,N/2$ belong to block $B_1$, while those with labels $N/2+1,\ldots,N$ belong to block $B_2$.
	
	Within each block, we consider a Curie--Weiss type interaction with inverse temperature $\beta>2$. In particular, we assume that every pair of vertices within the same block is connected by an edge, and that each such edge carries the corresponding interaction. By contrast, we assume that each vertex in $B_1$ is connected to exactly one vertex in $B_2$, and that the edges connecting the two blocks carry a weak Ising-type interaction whose strength vanishes as $N\to\infty$.
	
	More precisely, for $\beta>0$ and $0\le \alpha_N \le \beta$, we consider the Hamiltonian
	$$
	H_{N,\alpha_N, \beta}(\sigma):= -\frac \beta {2N} \sum_{i \sim j} \sigma_i \sigma_j
	-\alpha_N \sum_{i \in B_1} \sigma_i \sigma_{i+N/2},
	\quad \sigma \in \{-1,+1\}^N.
	$$
	Here, we write $i \sim j$ if and only if vertices $i$ and $j$ belong to the same block, i.e.\ $\sum_{i \sim j} \sigma_i \sigma_j:=\sum_{i,j\in B_1} \sigma_i \sigma_j+\sum_{i,j\in B_2} \sigma_i \sigma_j.$
	Note that $\alpha_N$ is allowed to depend on $N$; in fact, this dependence will turn out to be the most interesting regime.
	
	The Hamiltonian induces the Gibbs measure
	\begin{equation}\label{gibbs}
		\mu_{N, \alpha_N, \beta} (\sigma)
		:= \frac{e^{-H_{N,\alpha_N, \beta}(\sigma)}}{\sum_{\sigma'}e^{-H_{N,\alpha_N, \beta}(\sigma')}}=:
		\frac{e^{-H_{N,\alpha_N, \beta}(\sigma)}}{Z_{N, \alpha_N,\beta}}.
	\end{equation}
	
	Introducing the mean magnetizations per block,
	$$
	m_1:=m^N_1:=m_1(\sigma):= \frac {2} {N} \sum_{i \in B_1} \sigma_i
	\qquad \text{and} \quad
	m_2:=m^N_2:=m_2(\sigma):= \frac {2} {N} \sum_{i \in B_2} \sigma_i,
	$$
	we may rewrite the Hamiltonian as
	$$
	H_{N,\alpha_N, \beta}(\sigma)=
	-\frac {\beta N} 8 \left(m_1^2+ m_2^2\right)
	-\alpha_N \sum_{i \in B_1} \sigma_i \sigma_{i+N/2},
	\quad \sigma \in \{-1,+1\}^N.
	$$
	
	From now on, we assume that $\beta>2$ and that $\alpha_N>0$ converges to $0$.
	More specifically, for convenience, we take $\alpha_N = N^{-\rho}$ for some $\rho>0$.
	The restriction $\beta>2$ is motivated by the fact that, for this choice of $\beta$ and $\alpha_N\equiv 0$, the model is known to be in the low-temperature phase; see \cite{FU12, BRS19, LS18, KT18, KLSS19, LSV20}.
	This behavior can be understood by studying the asymptotics of the magnetization vector
	$m^N=(m_1,m_2)$.
	
	Indeed, even for the model with Hamiltonian
	$$
	\overline H_{N,\alpha, \beta}(\sigma):=
	-\frac N 2 \left(
	\frac{1}{2}\alpha m_1 m_2
	+\beta \frac{1}{4} m_1^2
	+\frac{1}{4} \beta m_2^2
	\right),
	$$
	(which we introduce here for reference),
	where $0 \le \alpha \le \beta$, one observes the following behavior.
	If $\alpha+\beta\le 2$, the vector $m^N$ converges in probability and almost surely to the zero vector.
	By contrast, for $\beta>2$ and $\alpha\neq 0$, the distribution of $m^N$ converges to
	\be \label{eq:LTlimit}
	\frac 12 \left(
	\delta_{(m^+(\frac{\alpha+\beta}2),m^+(\frac{\alpha+\beta}2))}
	+\delta_{(m^-(\frac{\alpha+\beta}2),m^-(\frac{\alpha+\beta}2))}
	\right).
	\ee
	
	Finally, if $\beta>2$ and $\alpha=0$, the limiting distribution of $m^N$ is a mixture of four Dirac measures,
	\be \label{eq:mixedlimit}
	\frac 14 \left(
	\delta_{(m^+(\frac{\beta}2),m^+(\frac{ \beta}2))}+
	\delta_{(m^+(\frac{ \beta}2),m^-(\frac{ \beta}2))}+
	\delta_{(m^-(\frac{ \beta}2),m^+(\frac{ \beta}2))}+
	\delta_{(m^-(\frac{ \beta}2),m^-(\frac{ \beta}2))}
	\right).
	\ee
	Here, for $\gamma>0$, the quantity $m^+(\gamma)$ denotes the largest solution of
	\begin{equation}
		z=\tanh(\gamma z),
	\end{equation}
	and $m^-(\gamma)=-m^+(\gamma)$.
	
	The main question addressed in this note is whether the presence of the interaction term
	$\alpha_N \sum_{i \in B_1} \sigma_i \sigma_{i+N/2}$ is felt in the limit $\alpha_N\to 0$, and if so, at which rate of convergence.
	Our answer is summarized in the following theorem.
	\begin{theorem}\label{theo:theo1}
		Consider the two-block spin model with $\alpha_N>0$ such that $\alpha_N\downarrow 0$.
		\begin{enumerate}
			\item If $\lim N\alpha_N = \infty$, then the distribution of $m^N$ under the Gibbs measure converges to the limiting distribution
			\be\label{eq:LTalpha0}
				\frac 12 \left(
			\delta_{(m^+(\frac{\beta}2),m^+(\frac{\beta}2))}
			+\delta_{(m^-(\frac{\beta}2),m^-(\frac{\beta}2))}
			\right),
			\ee
			 (note that this is the same as setting $\alpha=0$ in \eqref{eq:LTlimit}).
			\item If $\lim N\alpha_N = 0$, then the distribution of $m^N$ under the Gibbs measure converges to the limiting distribution given in \eqref{eq:mixedlimit}.
		\end{enumerate}
	\end{theorem}
	
	The theorem shows that even a sparse (one inter-block edge per spin) and weak (i.e., $\alpha_N\to 0$) interaction between the blocks affects the limiting behavior, provided it is not too weak, in the sense that $\lim N\alpha_N=\infty$.

    \subsection{A diluted matching}\label{sec:DiluteMatch}
	Recall that one of the main motivations for introducing the model in Section~\ref{sec:PerfectMatch} was to analyze how the insertion of bottlenecks into a complete graph affects the behavior of the Ising model defined on it. One might argue that a configuration in which each vertex in $B_1$ is connected to exactly one vertex in $B_2$ does not constitute a genuine bottleneck --- although, from the perspective of the complete graph, it does represent a severe restriction of connectivity.
	
	In this section, we consider a first variation of the model in which the connections between the two blocks are further thinned. More precisely, we assume that each potential edge between $i\in B_1$ and $i+\frac N2\in B_2$ is present independently with probability $p(N)$. Formally, this leads to the Hamiltonian
	\[
	H_{N,\alpha_N, \beta}^{\eps}(\sigma)
	:= -\frac \beta {2N} \sum_{i \sim j} \sigma_i \sigma_j
	-\alpha_N \sum_{i \in B_1} \varepsilon_{i} \sigma_i \sigma_{i+N/2},
	\quad \sigma \in \{-1,+1\}^N,
	\]
	where the superscript $\eps$ indicates that it depends on the value of the random variables $(\eps_i)$ that are Bernoulli random variables satisfying
	\[
	\P(\varepsilon_{i}=1)=1-\P(\varepsilon_{i}=0)=p(N).
	\]
	Let $\mu_{N,\alpha_N,\beta}^{\eps}$ denote the Gibbs measure associated with the Hamiltonian
	$H_{N,\alpha_N, \beta}^{\eps}$.
	We assume throughout that $N p(N) \to \infty$, since otherwise the number of inter-block edges does not diverge,
	so one cannot expect a collective alignment effect of the type studied in Theorem \ref{theo:theo1},
	and in particular the event of having no inter-block edges has non-vanishing probability. It turns out that this condition alone is not sufficient to enforce alignment: what matters is the combined scale $p(N)N\alpha_N$ of the inter-block
	interaction. The limiting behavior is therefore governed by the same mechanism as
	in Theorem~\ref{theo:theo1}, but with a renormalized effective interaction scale.

	\begin{theorem}\label{theo:theo2}
		Consider the two-block spin model with Hamiltonian $H_{N,\alpha_N, \beta}^{\eps}$ and parameters $\alpha_N>0$ satisfying $\alpha_N\downarrow 0$.
		\begin{enumerate}
			\item If $\lim p(N) N \alpha_N = \infty$, then the distribution of $m^N$ under $\mu_{N,\alpha_N,\beta}^{\eps}$ converges to the limiting distribution given in \eqref{eq:LTalpha0} in the sense of convergence in distribution in probability. \label{thm:diluteCase1}
			\item If $\lim p(N) N \alpha_N = 0$, then the distribution of $m^N$ under $\mu_{N,\alpha_N,\beta}^{\eps}$ converges to the limiting distribution given in \eqref{eq:mixedlimit} again in the sense of convergence in distribution in probability.\label{thm:diluteCase2}
		\end{enumerate}
        If we further assume that $p(N)N \ge C\log N$ for a constant $C>0$ large enough, then the convergences in (\ref{thm:diluteCase1}) and (\ref{thm:diluteCase2}) hold in the sense of convergence in distribution almost surely.
	\end{theorem}

    \subsection{A 3-block model}
    In this section we consider a version of a bottleneck spin model that is based on a block spin model with three blocks. More precisely, we consider a block spin model where the number of blocks is equal to three. We want to have two "main" blocks and between them a much smaller block interacting with both of them and serving as a bottleneck. Formally, denote the blocks by $B_1,B_2$ and $B_3$ and by $|B_k|$ denote the size of the $k$'th block. Assume that $|B_2| =b_N$ for a non-decreasing sequence $(b_N)_N\subset \N$ with $b_N \to \infty$ but $\frac{b_N}{N} \to 0$. Let $|B_1|= |B_3| = \frac{N-b_N}{2}$. We consider on each block a mean field interaction at inverse temperature $\beta >1$ and between the first and second block as well as between the third and second block at strength $\alpha_N$. (Note that we choose $\beta >1$ since due to the different normalization in the Hamiltonian defined below, the critical temperature in the case $\alpha_N\equiv 0$ is $1$ instead of $2$.)

More precisely, define the Hamiltonian as
\begin{flalign*}
	H_{N, \beta, \alpha_N} (\sigma)&= -\sum_{k=1}^3 \frac{\beta}{2}\frac{1}{|B_k|}\sum_{i,j\in B_k} \sigma_i\sigma_j - \alpha_N\frac{1}{\sqrt{|B_1||B_2|}} \sum_{\substack{i \in B_1 \\ j\in B_2}} \sigma_i\sigma_j - \alpha_N\frac{1}{\sqrt{|B_1||B_2|}} \sum_{\substack{i \in B_3 \\ j\in B_2}} \sigma_i\sigma_j \\
	&= -\sum_{k=1}^3\frac{\beta}{2}|B_k|m_k^2 - \alpha_N\sqrt{|B_1||B_2|}m_1m_2 - \alpha_N\sqrt{|B_3||B_2|}m_3m_2, 
\end{flalign*}
where $m_k=\frac{1}{|B_k|}\sum_{i\in B_k} \sigma_i$ is the average magnetization of the $k$'th block. In particular note that the Hamiltonian is a function of the magnetization, or in other words, all spin configurations that lead to the same magnetization vector have the same energy. Write $H_{N,\beta,\alpha_N}(m)$. Then the Gibbs measure on the spin configurations is given by
$$
\mu_{N,\beta,\alpha_N}( \sigma) = \frac{1}{Z_{N, \beta, \alpha_N}}\exp\{ -H_{N,\beta,\alpha_N}(\sigma)\}\frac{1}{2^N}
$$
and in terms of the magnetization
$$
\mu_{N,\beta, \alpha_N}( m) =  \frac{1}{Z_{N, \beta, \alpha_N}} \exp\{ -H_{N,\beta,\alpha_N}(m)\}\frac{1}{2^N} \prod_{k=1}^3 \binom{|B_k|}{|B_k|\frac{1+m_k}{2}},
$$
where
$$
Z_{N, \beta, \alpha_N} = \sum_{m_1,m_2,m_3}  \exp\{ -H_{N,\beta,\alpha_N}(m)\}\frac{1}{2^N} \prod_{k=1}^3 \binom{|B_k|}{|B_k|\frac{1+m_k}{2}}
$$
denotes the partition function of this model, and the factor $ \prod_{k=1}^3 \binom{|B_k|}{|B_k|\frac{1+m_k}{2}}$ is the number of spin configurations associated to the realization of the magnetization vector $(m_1,m_2,m_3)$.
Similarly to Sections \ref{sec:PerfectMatch} and \ref{sec:DiluteMatch}, there is a threshold for the interaction strength that determines whether or not the (three) blocks align in the thermodynamic limit. In between these two extreme cases, both alignment and anti-alignment have non-zero probability. In this model however, there is a second threshold for the interaction strength, that determines whether or not the magnetization of the small block $m_2$ is influenced by its larger neighbors. This influence is felt in the sense of an external magnetic field. Let us denote by $m^* = m^*(\beta)$, the largest solution to the Curie-Weiss equation
\begin{equation}\label{eq:CW}
    x = \tanh( \beta x).
\end{equation}
Further, for any $c \in [0, \infty)$, let us denote by $m(c)$, the largest solution to the equation
\begin{equation}
    x = \tanh( \beta x + \sqrt{2}cm^* ), \label{eq:CWMagnField}
\end{equation}
and define $m(\infty) := 1$. Note that $-m(c)=m(-c)$ is the smallest solution to $x = \tanh( \beta x - \sqrt{2}cm^* )$. In other words, $m(c)$ is the limiting magnetization of a standard Curie-Weiss model with external magnetic field $h= \sqrt{2}cm^*$. Further note that for $c=0$, we recover the equation in \eqref{eq:CW}, i.e.\ the case with no external magnetic field.

\begin{theorem}\label{thm:Blockspin}
	In the three block model described above with $\alpha_N >0$ such that $\alpha_N \downarrow 0$, the following convergences hold:
	\begin{enumerate}
		\item If $\lim_{N\to \infty} \alpha_N\sqrt{\frac{N}{b_N}} = \infty $, then the distribution of $m$ under the Gibbs measure converges weakly to 
		$$\frac{1}{2}\left(\delta_{(m^*,1,m^*)} + \delta_{(-m^*,-1,-m^*)} \right).$$ \label{thm:Blockspin1}
		\item If $\lim_{N\to \infty} \alpha_N\sqrt{\frac{N}{b_N}} = 0 $, but $\lim_{N\to \infty}\alpha_N\sqrt{b_N N}=\infty$, then the distribution of $m$ under the Gibbs measure converges weakly to 
		$$\frac{1}{2}\left(\delta_{(m^*,m^*,m^*)} + \delta_{(-m^*,-m^*,-m^*)} \right).$$ \label{thm:Blockspin2}
		\item If $\lim_{N\to \infty} \alpha_N\sqrt{b_N N} = 0 $,  then the distribution of $m$ under the Gibbs measure converges weakly to 
		$$\frac{1}{8}\sum_{\chi_1,\chi_2, \chi_3 \in \{-1,+1\}} \delta_{(\chi_1 m^*,\chi_2 m^*,\chi_3m^*)}.$$ \label{thm:Blockspin3}
	\end{enumerate}
\end{theorem}
The results in Theorem \ref{thm:Blockspin} describe the behavior of the magnetization in the extreme cases $\alpha_N\sqrt{\frac{N}{b_N}} \longrightarrow c \in \{0,\infty\}$, $\alpha_N\sqrt{b_N N} \longrightarrow C \in \{ 0,\infty\}$. The results for the intermediate regimes are captured in:
\begin{theorem}\label{thm:PhaseTransition}
In the three block model described above with $\alpha_N >0$ such that $\alpha_N \downarrow 0$, the following convergences hold:
\begin{enumerate}
    \item If $\lim_{N\to \infty} \alpha_N\sqrt{\frac{N}{b_N}} = c$ for some constant $c\in (0,\infty)$, then the distribution of $m$ under the Gibbs measure converges weakly to 
    $$\frac{1}{2}\left(\delta_{(m^*,m(c),m^*)} + \delta_{(-m^*,-m(c),-m^*)} \right).$$ \label{thm:Phasetrans1}
    \item If  $\lim_{N\to \infty}\alpha_N\sqrt{b_N N}=C$ for some constant $C \in (0,\infty)$, then the distribution of $m$ under the Gibbs measure converges weakly to
    $$\sum_{\chi_1,\chi_2, \chi_3 \in \{-1,+1\}} a(\chi_1,\chi_2,\chi_3,C)\delta_{(\chi_1 m^*,\chi_2 m^*,\chi_3m^*)},$$ \label{thm:Phasetrans2}
    where $a(\chi_1,\chi_2,\chi_3,C)$ are weight depending on $C$, that are given by
    \begin{align*}
        & a(\chi_1,\chi_2,\chi_3,C) = \frac{1}{2\left(  1+ e^{-\sqrt{2}C(m^*)^2} \right)^2}\quad \text{ if } \chi_1=\chi_2=\chi_3 \\
        &  a(\chi_1,\chi_2,\chi_3,C) = \frac{e^{-2\sqrt{2}C(m^*)^2}}{2\left(  1+ e^{-\sqrt{2}C(m^*)^2} \right)^2}\quad \text{ if } \chi_1=\chi_3=- \chi_2 \\
        &  a(\chi_1,\chi_2,\chi_3,C) = \frac{e^{-\sqrt{2}C(m^*)^2}}{2\left(  1+ e^{-\sqrt{2}C(m^*)^2} \right)^2}\quad \text{ if } \chi_1=\pm \chi_2=- \chi_3
    \end{align*}
    \end{enumerate}
\end{theorem}
\begin{remark}
    Note that in both results in Theorem \ref{thm:PhaseTransition}, if one considers $C \to 0$ or $C \to \infty$, and $c \to 0$ or $c \to \infty$ respectively, one recovers the results in Theorem \ref{thm:Blockspin}.
\end{remark}

	\section{Proof of Theorem \ref{theo:theo1}}
	
	The proof proceeds in three main steps:
	
	First, we show that under the Gibbs measure the magnetization vector
	$m^N=(m_1,m_2)$ concentrates on a small neighborhood $A_\kappa$ of the four
	low-temperature Curie--Weiss equilibria
	$(\pm m^*(\beta/2),\pm m^*(\beta/2))$. This is achieved by combining moderate
	deviation estimates for the Curie--Weiss model with a robustness argument
	showing that the weak inter-block interaction does not destroy this
	concentration provided $2\kappa<\rho$.
	
	Second, we analyze the inter-block interaction conditional on a fixed
	magnetization $(\mu_1,\mu_2)\in A_\kappa$. Conditioning on the magnetization,
	the Curie--Weiss part of the Hamiltonian becomes constant and the remaining
	randomness concerns the perfect matching between the two blocks. A large
	deviation analysis of this matching shows that, with overwhelming probability,
	the cross-block interaction term is asymptotically equal to
	$\frac N2\,\mu_1\mu_2$.
	
	Finally, we compare the Gibbs weights of the four wells in $A_\kappa$.
	While the Curie--Weiss contributions are asymptotically identical for all
	four sign combinations, the inter-block term favors aligned magnetizations
	$\mu_1\mu_2>0$ over anti-aligned ones $\mu_1\mu_2<0$. Depending on whether
	$N\alpha_N\to\infty$ or $N\alpha_N\to0$, this bias either dominates or vanishes,
	leading respectively to a two-point or a four-point limiting distribution for
	$m^N$.
	
	\medskip
	We begin the formal proof of Theorem~\ref{theo:theo1} by identifying regions in $[-1,1]^2$ that cannot arise as limit points of the magnetization vector $m^N$ under the assumptions of the theorem. To this end, for $0<\kappa<\frac12$, $\beta>2$, and $N\in 2\N$, we define the sets
	\begin{flalign} \label{eq:Agamma}
	A_\kappa:=A_\kappa^N  \nonumber
    :=\Bigg\{\mu\in [-1,1]^2: \mu \in \bigcup_{s_1,s_2 \in\{-,+\}}
	&\left(m^{s_1}\left(\frac \beta2 \right)-N^{-\kappa},\,m^{s_1}\left(\frac \beta 2\right)+N^{-\kappa}\right) \\
    &\quad \times\left(m^{s_2}\left(\frac \beta2 \right)-N^{-\kappa},\,m^{s_2}\left(\frac \beta 2\right)+N^{-\kappa}\right) \Bigg\}
	\end{flalign}

	\begin{proposition}\label{prop:atypical_mu}
		If $\alpha_N = N^{-\rho}\to 0$ for some $\rho>0$, then for the set $A_\kappa^N$ defined above,
		\[
		\mu_{N,\alpha_N,\beta}\bigl((A_\kappa^N)^c\bigr)\to 0
		\]
		provided that $2\kappa<\rho$.
	\end{proposition}
	
	\begin{proof}
			We first consider the decoupled case $\alpha_N \equiv 0$. In that case, the
			Hamiltonian splits into two independent Curie--Weiss Hamiltonians on the blocks
			$B_1$ and $B_2$, each with $N/2$ spins and effective inverse temperature
			$\beta/2>1$. Hence, by the standard moderate deviation principle for the
			low-temperature Curie--Weiss model (see, e.g., \cite{EL04}), there exist
			constants $c,C>0$ such that
			\begin{multline*}
			\mu_{N,0,\beta}\Bigl(\bigl|m_i-m^+(\beta/2)\bigr|>N^{-\kappa},
			\ m_i\ge 0\Bigr)
			+
			\mu_{N,0,\beta}\Bigl(\bigl|m_i-m^-(\beta/2)\bigr|>N^{-\kappa},
			\ m_i\le 0\Bigr)
		\\	\le C e^{-c N^{1-2\kappa}}
			\end{multline*}
			for $i=1,2$ and all $N$ sufficiently large.
			
			Since each block magnetization takes at most $N/2+1$ values, a union bound
			yields
			\begin{equation}\label{eq:MDP_decoupled}
				\mu_{N,0,\beta}\bigl((A_\kappa^N)^c\bigr)
				\le C' (N+1)^2 e^{-c N^{1-2\kappa}}
			\end{equation}
			for some constant $C'>0$.
			
			We now compare the coupled measure $\mu_{N,\alpha_N,\beta}$ with the decoupled
			measure $\mu_{N,0,\beta}$. For every configuration $\sigma\in\{-1,+1\}^N$,
			\[
			\left|H_{N,\alpha_N,\beta}(\sigma)-H_{N,0,\beta}(\sigma)\right|
			=
			\alpha_N\left|\sum_{i\in B_1}\sigma_i\sigma_{i+N/2}\right|
			\le \frac N2 \alpha_N.
			\]
			Therefore,
			\[
			e^{-N\alpha_N/2} e^{-H_{N,0,\beta}(\sigma)}
			\le
			e^{-H_{N,\alpha_N,\beta}(\sigma)}
			\le
			e^{N\alpha_N/2} e^{-H_{N,0,\beta}(\sigma)}.
			\]
			Summing over all $\sigma$ gives
			\[
			e^{-N\alpha_N/2} Z_{N,0,\beta}
			\le
			Z_{N,\alpha_N,\beta}
			\le
			e^{N\alpha_N/2} Z_{N,0,\beta}.
			\]
			Hence, for every event $E\subset\{-1,+1\}^N$,
			\[
			\mu_{N,\alpha_N,\beta}(E)
			=
			\frac{\sum_{\sigma\in E} e^{-H_{N,\alpha_N,\beta}(\sigma)}}{Z_{N,\alpha_N,\beta}}
			\le
			e^{N\alpha_N}\,
			\frac{\sum_{\sigma\in E} e^{-H_{N,0,\beta}(\sigma)}}{Z_{N,0,\beta}}
			=
			e^{N\alpha_N}\mu_{N,0,\beta}(E).
			\]
			Applying this with $E=(A_\kappa^N)^c$ and using \eqref{eq:MDP_decoupled}, we obtain
			\[
			\mu_{N,\alpha_N,\beta}\bigl((A_\kappa^N)^c\bigr)
			\le
			C'(N+1)^2
			\exp\bigl(-cN^{1-2\kappa}+N\alpha_N\bigr).
			\]
			Since $\alpha_N=N^{-\rho}$, we have $N\alpha_N=N^{1-\rho}$, and the assumption
			$2\kappa<\rho$ implies
			\[
			N^{1-\rho}=o\bigl(N^{1-2\kappa}\bigr).
			\]
			Therefore, the right-hand side converges to $0$ as $N\to\infty$, proving the
			claim.
\end{proof}		
	
	From now on, we fix $\kappa$ such that $2\kappa<\rho$ and refer to the elements of the set $A_\kappa$ defined in \eqref{eq:Agamma} as \emph{typical points}. For a given $N$, a typical element $(\mu_1,\mu_2)\in A_\kappa$ is called \emph{admissible} if both $N(1+\mu_1)/4$ and $N(1+\mu_2)/4$ are natural numbers. This condition ensures that configurations exist with exactly $N(1+\mu_1)/4$ and $N(1+\mu_2)/4$ plus spins in blocks $B_1$ and $B_2$, respectively, yielding $m_1=\mu_1$ and $m_2=\mu_2$. In order to investigate whether the two blocks align in the thermodynamic limit or not, we need a more thorough analysis of the inter-block interaction term. To that end, we condition on a block magnetization $(\mu_1,\mu_2)\in A_\kappa$. By computing the typical number of matched spin pairs that are aligned/anti-aligned we show:

    	\begin{lemma}\label{lem:cross_conc}
		Fix typical and admissible $(\mu_1,\mu_2)\in A_\kappa$, and let
		$\nu_N^{\mu_1,\mu_2}$ denote the uniform probability measure on
		\[
		\Omega_N(\mu_1,\mu_2):=\{\sigma\in\{-1,+1\}^N:\ m_1(\sigma)=\mu_1,\ m_2(\sigma)=\mu_2\}.
		\]
		Then for every $\varepsilon>0$ there exists $c_\varepsilon>0$ such that
		\[
		\nu_N^{\mu_1,\mu_2}\Big(\Big|\frac1N\sum_{i\in B_1}\sigma_i\sigma_{i+N/2}
		-\frac12\mu_1\mu_2\Big|>\varepsilon\Big)\le e^{-c_\varepsilon N}
		\]
		for all $N$ large enough.
	\end{lemma}
    
\begin{remark}As a corollary of Lemma~\ref{lem:cross_conc} we obtain: 
	
	Under the assumptions of Lemma~\ref{lem:cross_conc},
	\[
	n_{++}=\frac N8(1+\mu_1)(1+\mu_2)+o(N),
	\qquad
	n_{--}=\frac N8(1-\mu_1)(1-\mu_2)+o(N)
	\]
	with overwhelming probability.
\end{remark}
	
	\begin{proof}
	    
	Let $\mu=(\mu_1,\mu_2)\in A_\kappa$ be typical and admissible. Then there are $N(1+\mu_1)/4$ plus spins in $B_1$ and $N(1+\mu_2)/4$ plus spins in $B_2$. The probability that exactly $n$ of the plus spins in $B_1$ are paired with plus spins in $B_2$ is given by
	\be \label{eq:wkeit_typ}
	\frac{\binom{\frac{N}2}{n}\binom{\frac{N}2-n}{\frac{N}4(1+\mu_1)-n}
		\binom{\frac N2-\frac{N}4(1+\mu_1)}{\frac{N}4(1+\mu_2)-n}}
	{\binom{\frac N2}{\frac{N}4(1+\mu_1)}\binom{\frac N2}{\frac{N}4(1+\mu_2)}}.
	\ee
	Indeed, there are $\binom{\frac{N}2}{n}$ choices for the $+{+}$ pairs across the two blocks. The remaining $\frac{N}4(1+\mu_1)-n$ plus spins in $B_1$ can be placed in $\binom{\frac{N}2-n}{\frac{N}4(1+\mu_1)-n}$ ways. Finally, the remaining $\frac{N}4(1+\mu_2)-n$ plus spins in $B_2$ must be placed among the $\frac N2-\frac{N}4(1+\mu_1)$ vertices not paired with plus spins in $B_1$, yielding $\binom{\frac N2-\frac{N}4(1+\mu_1)}{\frac{N}4(1+\mu_2)-n}$ possibilities.
	
	In total, there are $\binom{\frac N2}{\frac{N}4(1+\mu_1)}\binom{\frac N2}{\frac{N}4(1+\mu_2)}$ admissible configurations of plus spins in blocks $B_1$ and $B_2$. We note that this expression is well defined provided
	\[
	\frac{N}4(1+\mu_1)+\frac{N}4(1+\mu_2)-n\le \frac N2,
	\]
	an assumption that we will impose throughout.

Since for typical and admissible $(\mu_1,\mu_2)$ both
$\frac N4(1+\mu_1)$ and $\frac N4(1+\mu_2)$ are of order $N$,
also $n$ will typically be of order $N$. Write $n=\gamma \frac N4$
with $\gamma\in\R$ such that $\gamma\frac N4\in\N$.
Let
\[
a:=\frac N4(1+\mu_1)\qquad \text{and }\quad b:=\frac N4(1+\mu_2).
\]
The number $n$ of $++$-connections must satisfy
\[
\max\Big\{0,\,a+b-\frac N2\Big\}\ \le\ n\ \le\ \min\{a,b\},
\]
equivalently,
\begin{equation}\label{eq:gamma_range}
	\gamma\in\Big[\max\{0,\mu_1+\mu_2\},\ \min\{1+\mu_1,1+\mu_2\}\Big]
\end{equation}
( and 
$\gamma \frac N 4$ needs to be an integer).

 We will therefore expand the numerator  \eqref{eq:wkeit_typ} using Stirling's formula. To this end notice that 
	\begin{multline*}
		\binom{\frac{N}2}{n}\binom{\frac{N}2-n}{\frac{N}4(1+\mu_1)-n}\binom{\frac N2-\frac{N}4(1+\mu_1)}{\frac{N}4(1+\mu_2)-n}\\=
		\frac{\frac{N}2 !}{n!(\frac{N}4(1+\mu_1)-n)!(\frac{N}4(1+\mu_2)-n)!
			(\frac N2 -\frac{N}4(1+\mu_1)-\frac{N}4(1+\mu_2)+n)!}.
	\end{multline*}
	Employing Stirling's formula to the factorials on the right hand side yields
	\begin{align*}
		&\frac{\frac{N}2 !}{n!(\frac{N}4(1+\mu_1)-n)!(\frac{N}4(1+\mu_2)-n)!
			(\frac N2 -\frac{N}4(1+\mu_1)-\frac{N}4(1+\mu_2)+n)!}\\
		=&\frac{\sqrt {\frac N2}}{(2 \pi)^{3/2}\sqrt{n(\frac{N}4(1+\mu_1)-n)(\frac{N}4(1+\mu_2)-n)(\frac N2 -\frac{N}4(1+\mu_1)-\frac{N}4(1+\mu_2)+n)}}	\\
		&\times \exp\left(\frac N2 \log \frac N2-n \log n-\big(\frac{N}4(1+\mu_1)-n\big)\log\big(\frac{N}4(1+\mu_1)-n\big)\right.\\
		&\qquad \qquad \left.-
		\big(\frac{N}4(1+\mu_2)-n\big)\log\big(\frac{N}4(1+\mu_2)-n\big) \right.\\
		&\qquad \qquad \qquad \left.
		-
		\big(\frac N2 -\frac{N}4(1+\mu_1) -\frac{N}4(1+\mu_2)+n\big)\log\big(\frac N2 -\frac{N}4(1+\mu_1)-\frac{N}4(1+\mu_2)+n\big)\right)
	\end{align*}
	Our focus will be on the $\gamma$-dependent part of 
	the exponent. Let us abbreviate this by
	\begin{align*}
		S(\gamma)&:=-n \log n-\big(\frac{N}4(1+\mu_1)-n\big)\log\big(\frac{N}4(1+\mu_1)-n\big)-
		\big(\frac{N}4(1+\mu_2)-n\big)\log\big(\frac{N}4(1+\mu_2)-n\big) \\
		&\qquad \qquad
		-
		\big(\frac N2 -\frac{N}4(1+\mu_1) -\frac{N}4(1+\mu_2)+n\big)\log\big(\frac N2 -\frac{N}4(1+\mu_1)-\frac{N}4(1+\mu_2)+n\big) 	\\
		&= - m \Big(\gamma \log (\gamma m)+(1+\mu_1-\gamma)
		\log(m(1+\mu_1-\gamma))+(1+\mu_2-\gamma)
		\log(m(1+\mu_2-\gamma))\Big.\\
		&\qquad \Big. + (\gamma-(\mu_1+\mu_2))\log(m(\gamma-(\mu_1+\mu_2)))\Big)
	\end{align*}
	where we have set $m:=\frac N 4$ and used that $n=
	\gamma m$.
	
Ignoring additive constants that do not depend on $\gamma$, we may write
\begin{multline*}
\frac1m S(\gamma)\\=
-\Big[\gamma\log\gamma+(1+\mu_1-\gamma)\log(1+\mu_1-\gamma)
+(1+\mu_2-\gamma)\log(1+\mu_2-\gamma)
+(\gamma-(\mu_1+\mu_2))\log(\gamma-(\mu_1+\mu_2))\Big] + C,
\end{multline*}
where $C$ does not depend on $\gamma$.
Differentiating yields
\begin{align*}
	\frac{d}{d\gamma}S(\gamma)
	&=-m\Big(\log\gamma+\log(\gamma-(\mu_1+\mu_2))
	-\log(1+\mu_1-\gamma)-\log(1+\mu_2-\gamma)\Big).
\end{align*}
	Moreover,
	\[
	\frac {d^2}{d\gamma^2}S(\gamma)=
	-m\left(\frac 1 \gamma+\frac 1{\gamma-(\mu_1+\mu_2)}
	+\frac 1{1+\mu_1-\gamma}+\frac 1{1+\mu_2-\gamma}\right).
	\]
	By \eqref{eq:gamma_range}, for $\gamma$ in the interior of the feasible interval
	we have
	\[
	\gamma>0,\quad \gamma-(\mu_1+\mu_2)>0,\quad 1+\mu_1-\gamma>0,\quad 1+\mu_2-\gamma>0,
	\]
	so all four denominators are positive. Hence the bracket is strictly positive
	and therefore $S''(\gamma)<0$. In particular, $S$ is strictly concave on the
	feasible interval, so the maximizer $\gamma^*$ is unique and is characterized
	by $S'(\gamma^*)=0$.
 Such a $\gamma^*$ satisfies
	$$
	\log \frac{\gamma^*(\gamma^*-(\mu_1+\mu_2))}
	{(1+\mu_1-\gamma^*)(1+\mu_2-\gamma^*)}=0,
	$$
	i.e.\
	$$
	\gamma^*=\frac 12 (\mu_1+\mu_2+\mu_1\mu_2+1).
	$$
	Notice that for given $(\mu_1, \mu_2)\in A_\kappa$ configurations with $\gamma' \frac N4$ (where $\gamma' \neq \gamma^*$) many $++$ connections 
	between $B_1$ and $B_2$ have an asymptotically
	exponentially smaller probability than those with    
	$\gamma^* \frac N4$ many $++$ connections 
	between $B_1$ and $B_2$. 
	
	Following the same arguments for the minus-minus 
	connections we see that with overwhelming probability there are $\gamma^{**}\left( \frac{N}{4} + o(N) \right)$ many $--$ connections between $B_1$ and $B_2$, where
    $$
    \gamma^{**} = \frac{1}{2}\left( - \mu_1 - \mu_2 + \mu_1 \mu_2 + 1\right).
    $$
    In other words, there are $(\gamma^* + \gamma^{**})\frac{N}{4} = (\mu_1\mu_2 + 1)\frac{N}{4} + o(N)$ many aligned pairs of spins and therefore $\frac{N}{2}-(\mu_1\mu_2 + 1)\frac{N}{4} + o(N)$ many spin pairs of opposite orientation. This finished the proof of the Lemma.

    \end{proof}

For an admissible magnetization vector $\mu=(\mu_1,\mu_2)$, let
\[
\Omega_N(\mu):=\{\sigma\in\{-1,+1\}^N:\ m_1(\sigma)=\mu_1,\ m_2(\sigma)=\mu_2\}
\]
and let $\nu_N^\mu$ denote the uniform probability measure on $\Omega_N(\mu)$.
Moreover, define
\[
S(\sigma):=\sum_{i\in B_1}\sigma_i\sigma_{i+N/2}.
\]
Since the Curie--Weiss part of the Hamiltonian is constant on $\Omega_N(\mu)$, we have
\begin{equation}\label{eq:mu-decomp}
	\mu_{N,\alpha_N,\beta}(m^N=\mu)
	=
	\frac{1}{Z_{N,\alpha_N,\beta}}
	\exp\!\left(\frac{\beta N}{8}(\mu_1^2+\mu_2^2)\right)
	|\Omega_N(\mu)|
	\mathbb E_{\nu_N^\mu}\!\left[e^{\alpha_N S}\right].
\end{equation}

\medskip

	\begin{corollary}\label{cor:tilted_expectation}
		Let $\mu=(\mu_1,\mu_2)\in A_\kappa$ be typical and admissible, and define
		$
		S(\sigma)$ as above. 
		Then for every $\delta>0$,
		\[
		\exp\!\left(\frac{\alpha_N N}{2}\mu_1\mu_2-\delta \alpha_N N\right)
		\bigl(1-e^{-c_\delta N}\bigr)
		\le
		\mathbb E_{\nu_N^\mu}\!\left[e^{\alpha_N S}\right]
		\]
		and
		\[
		\mathbb E_{\nu_N^\mu}\!\left[e^{\alpha_N S}\right]
		\le
		\exp\!\left(\frac{\alpha_N N}{2}\mu_1\mu_2+\delta \alpha_N N\right)
		+
		e^{\alpha_N N/2-c_\delta N}
		\]
		for all $N$ sufficiently large and some constant $c_\delta>0$.
		In particular,
		\begin{equation}\label{eq:tilted_expectation_asymp}
			\log \mathbb E_{\nu_N^\mu}\!\left[e^{\alpha_N S}\right]
			=
			\frac{\alpha_N N}{2}\mu_1\mu_2+o(\alpha_N N).
		\end{equation}
		Moreover, the estimate \eqref{eq:tilted_expectation_asymp} holds uniformly over all
		typical and admissible $\mu\in A_\kappa$.
	\end{corollary}

	\begin{proof}
		Fix $\delta>0$ and let
		\[
		G_{\mu,\delta}
		:=
		\left\{
		\sigma\in\Omega_N(\mu):
		\left|\frac1N S(\sigma)-\frac12\mu_1\mu_2\right|\le \delta
		\right\}.
		\]
		By Lemma~\ref{lem:cross_conc}, there exists $c_\delta>0$ such that
		\[
		\nu_N^\mu(G_{\mu,\delta}^c)\le e^{-c_\delta N}
		\]
		for all $N$ sufficiently large.
		
		On $G_{\mu,\delta}$ we have
		\[
		\exp\!\left(\alpha_N N\left(\frac12\mu_1\mu_2-\delta\right)\right)
		\le e^{\alpha_N S}
		\le
		\exp\!\left(\alpha_N N\left(\frac12\mu_1\mu_2+\delta\right)\right).
		\]
		Therefore,
		\begin{align*}
			\mathbb E_{\nu_N^\mu}\!\left[e^{\alpha_N S}\right]
			&\ge
			\mathbb E_{\nu_N^\mu}\!\left[e^{\alpha_N S}\mathbf 1_{G_{\mu,\delta}}\right] \\
			&\ge
			\exp\!\left(\alpha_N N\left(\frac12\mu_1\mu_2-\delta\right)\right)
			\nu_N^\mu(G_{\mu,\delta}),
		\end{align*}
		which proves the lower bound.
		
		Similarly,
		\begin{align*}
			\mathbb E_{\nu_N^\mu}\!\left[e^{\alpha_N S}\right]
			&\le
			\mathbb E_{\nu_N^\mu}\!\left[e^{\alpha_N S}\mathbf 1_{G_{\mu,\delta}}\right]
			+
			\mathbb E_{\nu_N^\mu}\!\left[e^{\alpha_N S}\mathbf 1_{G_{\mu,\delta}^c}\right] \\
			&\le
			\exp\!\left(\alpha_N N\left(\frac12\mu_1\mu_2+\delta\right)\right)
			+
			e^{\alpha_N N/2}\nu_N^\mu(G_{\mu,\delta}^c),
		\end{align*}
		which yields the upper bound.

	To see \eqref{eq:tilted_expectation_asymp} 
	note that from the above we have that, for every $\delta>0$ and all $N$ sufficiently large,
	\[
	\frac{\alpha_N N}{2}\mu_1\mu_2-\delta \alpha_N N + \log(1-e^{-c_\delta N})
	\le
	\log \mathbb E_{\nu_N^\mu}\!\left[e^{\alpha_N S}\right]
	\]
	and
	\[
	\log \mathbb E_{\nu_N^\mu}\left[e^{\alpha_N S}\right]
	\le
	\frac{\alpha_N N}{2}\mu_1\mu_2+\delta \alpha_N N
	+
	\log\left(
	1+\exp\left(\alpha_N N\Bigl(\tfrac12-\tfrac12\mu_1\mu_2+\delta\Bigr)-c_\delta N\right)
	\right).
	\]
	Since $\alpha_N N=o(N)$, the last logarithmic term is $o(1)$, uniformly over all
	typical and admissible $\mu\in A_\kappa$. Hence, for every $\delta>0$,
	\[
	\limsup_{N\to\infty}
	\sup_{\mu\in A_\kappa}
	\frac{1}{\alpha_N N}
	\left|
	\log \mathbb E_{\nu_N^\mu}\!\left[e^{\alpha_N S}\right]
	-
	\frac{\alpha_N N}{2}\mu_1\mu_2
	\right|
	\le \delta.
	\]
	Since $\delta>0$ was arbitrary, we conclude that \eqref{eq:tilted_expectation_asymp}
holds uniformly over all typical and admissible $\mu\in A_\kappa$.
	
	\end{proof}

\begin{proof}[Proof of Theorem~\ref{theo:theo1}]
	By Proposition~\ref{prop:atypical_mu},
	\[
	\mu_{N,\alpha_N,\beta}\bigl(m^N\notin A_\kappa\bigr)\longrightarrow 0,
	\]
	so it suffices to analyze the Gibbs weights of admissible magnetization vectors
	$\mu=(\mu_1,\mu_2)\in A_\kappa$.
	
	Fix such a $\mu$. By \eqref{eq:mu-decomp} and Corollary~\ref{cor:tilted_expectation},
	\begin{equation}\label{eq:mu_weight_asymp}
		\mu_{N,\alpha_N,\beta}(m^N=\mu)
		=
		\frac{|\Omega_N(\mu)|}{Z_{N,\alpha_N,\beta}}
		\exp\!\left(
		\frac{\beta N}{8}(\mu_1^2+\mu_2^2)
		+
		\frac{\alpha_N N}{2}\mu_1\mu_2
		+o(\alpha_N N)
		\right).
	\end{equation}
	
	Now let $\mu=(\mu_1,\mu_2)\in A_\kappa$ and $\tilde \mu=(\mu_1,-\mu_2)\in A_\kappa$.
	Since
	\[
	|\Omega_N(\mu)|
	=
	\binom{N/2}{\frac N4(1+\mu_1)}
	\binom{N/2}{\frac N4(1+\mu_2)}
	\]
	and
	\[
	|\Omega_N(\tilde\mu)|
	=
	\binom{N/2}{\frac N4(1+\mu_1)}
	\binom{N/2}{\frac N4(1-\mu_2)},
	\]
	we have $|\Omega_N(\mu)|=|\Omega_N(\tilde\mu)|$ by the symmetry
	$\binom{n}{k}=\binom{n}{n-k}$. Moreover,
	\[
	\mu_1^2+\mu_2^2=\mu_1^2+(-\mu_2)^2.
	\]
	Hence \eqref{eq:mu_weight_asymp} yields
	\[
	\frac{\mu_{N,\alpha_N,\beta}(m^N=\mu)}
	{\mu_{N,\alpha_N,\beta}(m^N=\tilde\mu)}
	=
	\exp\!\left(\alpha_N N\,\mu_1\mu_2 + o(\alpha_N N)\right).
	\]

	For $\mu\in A_\kappa$, we have
	\[
	\mu_i=\pm m^*(\beta/2)+O(N^{-\kappa}),\qquad i=1,2.
	\]
	Hence
	\[
	\mu_1\mu_2=
	\begin{cases}
		(m^*(\beta/2))^2+O(N^{-\kappa}), & \text{if }\operatorname{sign}(\mu_1)=\operatorname{sign}(\mu_2),\\[1ex]
		-(m^*(\beta/2))^2+O(N^{-\kappa}), & \text{if }\operatorname{sign}(\mu_1)=-\operatorname{sign}(\mu_2).
	\end{cases}
	\]
	Therefore, uniformly over admissible $\mu\in A_\kappa$,
	\[
	\frac{\mu_{N,\alpha_N,\beta}(m^N=(\mu_1,\mu_2))}
	{\mu_{N,\alpha_N,\beta}(m^N=(\mu_1,-\mu_2))}
	=
	\exp\!\left(\alpha_N N (m^*(\beta/2))^2 + o(\alpha_N N)\right)
	\]
	if $\operatorname{sign}(\mu_1)=\operatorname{sign}(\mu_2)$, whereas
	\[
	\frac{\mu_{N,\alpha_N,\beta}(m^N=(\mu_1,\mu_2))}
	{\mu_{N,\alpha_N,\beta}(m^N=(\mu_1,-\mu_2))}
	=
	\exp\!\left(-\alpha_N N (m^*(\beta/2))^2 + o(\alpha_N N)\right)
	\]
	if $\operatorname{sign}(\mu_1)=-\operatorname{sign}(\mu_2)$.
	
	Now partition $A_\kappa$ into the four sign-wells
	$
	A_\kappa^{++},A_\kappa^{+-}, A_\kappa^{-+},$ and $A_\kappa^{--}$,
	according to the signs of $\mu_1$ and $\mu_2$. By Proposition~\ref{prop:atypical_mu},
	it is enough to compare the total Gibbs masses of these four sets.
	
	By global spin-flip symmetry,
	\[
	\mu_{N,\alpha_N,\beta}(m^N\in A_\kappa^{++})
	=
	\mu_{N,\alpha_N,\beta}(m^N\in A_\kappa^{--})
	\]
	and
	\[
	\mu_{N,\alpha_N,\beta}(m^N\in A_\kappa^{+-})
	=
	\mu_{N,\alpha_N,\beta}(m^N\in A_\kappa^{-+}).
	\]
	Moreover, the number of admissible magnetization values in each well is at most polynomial in $N$.
	
	If $N\alpha_N\to\infty$, then the exponential factor
	$
	\exp\!\left(\alpha_N N (m^*(\beta/2))^2 + o(\alpha_N N)\right)
	$
	dominates any polynomial correction. Hence the total Gibbs mass of the aligned wells
	$A_\kappa^{++}\cup A_\kappa^{--}$ dominates that of the anti-aligned wells
	$A_\kappa^{+-}\cup A_\kappa^{-+}$.
	
	If $N\alpha_N\to 0$, then the above ratio tends to $1$ uniformly over admissible
	magnetizations in $A_\kappa$. Therefore corresponding magnetizations in the four wells
	have asymptotically equal Gibbs weights, and by symmetry the total Gibbs mass is
	asymptotically equally distributed among the four wells.

	If $N\alpha_N\to\infty$, the aligned wells dominate and the distribution of
	$m^N$ converges to
	\[
	\frac12\Bigl(
	\delta_{(m^*(\beta/2),m^*(\beta/2))}
	+
	\delta_{(-m^*(\beta/2),-m^*(\beta/2))}
	\Bigr).
	\]
	If $N\alpha_N\to 0$, the ratio tends to $1$, and by symmetry all four wells carry
	equal asymptotic mass. Hence the distribution of $m^N$ converges to
	\[
	\frac14\Bigl(
	\delta_{(m^*(\beta/2),m^*(\beta/2))}
	+\delta_{(m^*(\beta/2),-m^*(\beta/2))}
	+\delta_{(-m^*(\beta/2),m^*(\beta/2))}
	+\delta_{(-m^*(\beta/2),-m^*(\beta/2))}
	\Bigr).
	\]
	
	This proves the theorem.

	\end{proof}

\section{Proof of Theorem \ref{theo:theo2}}

	Recall that in Theorem \ref{theo:theo2}, we are considering a natural variant of the model considered in Theorem \ref{theo:theo1} that is obtained by first fixing an arbitrary perfect
	matching between the two blocks $B_1$ and $B_2$, and then independently retaining
	each matching edge with probability $p(N)\in(0,1]$.
	Equivalently, the inter-block interaction term in the Hamiltonian becomes
	\[
	-\alpha_N \sum_{i\in B_1}\varepsilon_i\,\sigma_i\sigma_{i+N/2},
	\]
	where $(\varepsilon_i)_{i\in B_1}$ are i.i.d.\ Bernoulli random variables with
	parameter $p(N)$.
	Let $\mu^{\varepsilon}_{N,\alpha_N,\beta}$ denote the corresponding \emph{quenched}
	Gibbs measure.
	
	Let $M_N:=\sum_{i\in B_1}\varepsilon_i$ be the number of retained inter-block edges.
	If $p(N)N\to\infty$, then $M_N=\frac N2 p(N)+o(Np(N))$ with high probability, so the
	inter-block interaction has a deterministic first-order scale.
	In this regime, the analysis of the proof of Theorem~\ref{theo:theo1} carries over
	with only minor modifications.
	
\medskip  

	Fix an arbitrary perfect matching between $B_1$ and $B_2$ and let
	$\varepsilon=(\varepsilon_i)_{i\in B_1}$ be i.i.d.\ Bernoulli$(p(N))$ variables.
	Consider the quenched Hamiltonian
	\[
	H_{N,\alpha_N,\beta}^\varepsilon(\sigma)
	= -\frac{\beta N}{8}\bigl(m_1(\sigma)^2+m_2(\sigma)^2\bigr)
	-\alpha_N\sum_{i\in B_1}\varepsilon_i\,\sigma_i\sigma_{i+N/2},
	\qquad \sigma\in\{-1,+1\}^N,
	\]
	and write $\mu_{N,\alpha_N,\beta}^\varepsilon$ for the corresponding quenched Gibbs
	measure. Let
	\[
	M_N:=\sum_{i\in B_1}\varepsilon_i
	\]
	be the number of retained inter-block edges.
\begin{proposition}\label{prop:atypical_mu_thinned}
	Assume that $\alpha_N p(N)=N^{-\rho}\to 0$ for some $\rho>0$ and that
	$2\kappa<\rho$.
	
	\begin{enumerate}
		\item If $Np(N)\to\infty$, then
		\[
		\mu_{N,\alpha_N,\beta}^\varepsilon(m^N\notin A_\kappa)\longrightarrow 0
		\]
		with $\varepsilon$-probability $1-o(1)$.
		
		\item If, in addition, there exists $\delta\in(0,1)$ such that
		\be \label{eq:cond_BC}
		\sum_{N=1}^\infty \exp\!\left(-\frac{\delta^2}{6}Np(N)\right)<\infty,
		\ee
		then
		\[
		\mu_{N,\alpha_N,\beta}^\varepsilon(m^N\notin A_\kappa)\longrightarrow 0
		\qquad\text{for }\mathbb P_\varepsilon\text{-a.e. }\varepsilon.
		\]
		In particular, the latter holds if
		$
		Np(N)\ge C\log N
		$
		for all sufficiently large $N$ and some sufficiently large constant $C>0$.
	\end{enumerate}
\end{proposition}

\begin{proof}
	Let
	\[
	M_N:=\sum_{i\in B_1}\varepsilon_i.
	\]
	Since $M_N\sim \mathrm{Bin}(N/2,p(N))$, the  Chernoff bound yields that
	for every $\delta\in(0,1)$,
	\[
	\mathbb P_\varepsilon\!\left(\left|M_N-\frac N2 p(N)\right|>\delta\frac N2 p(N)\right)
	\le
	2\exp\!\left(-\frac{\delta^2}{6}Np(N)\right).
	\]
	
	Fix such a $\delta\in(0,1)$.
	
	If $Np(N)\to\infty$, it follows that with $\varepsilon$-probability tending to one,
	\begin{equation}\label{eq:MN_good_event}
		(1-\delta)\frac N2 p(N)\le M_N\le (1+\delta)\frac N2 p(N).
	\end{equation}
	If, moreover, \eqref{eq:cond_BC} holds
	then by the Borel--Cantelli lemma, \eqref{eq:MN_good_event} holds for all sufficiently
	large $N$, for $\mathbb P_\varepsilon$-almost every realization.
	
	On the event \eqref{eq:MN_good_event}, we have
	\[
	\alpha_N M_N\le (1+\delta)\frac N2 \alpha_N p(N).
	\]
	Since $\alpha_N p(N)=N^{-\rho}$ and $2\kappa<\rho$, it follows that
	$
	\alpha_N M_N=o(N^{1-2\kappa}).
	$
	
	Now the proof of Proposition~\ref{prop:atypical_mu} carries over verbatim.
	Indeed, for every configuration $\sigma\in\{-1,+1\}^N$,
	\[
	\left|H_{N,\alpha_N,\beta}^\varepsilon(\sigma)-H_{N,0,\beta}(\sigma)\right|
	=
	\alpha_N\left|\sum_{i\in B_1}\varepsilon_i\,\sigma_i\sigma_{i+N/2}\right|
	\le \alpha_N M_N.
	\]
	Therefore,
	\[
	e^{-\alpha_N M_N}e^{-H_{N,0,\beta}(\sigma)}
	\le
	e^{-H_{N,\alpha_N,\beta}^\varepsilon(\sigma)}
	\le
	e^{\alpha_N M_N}e^{-H_{N,0,\beta}(\sigma)}.
	\]
	Summing over all $\sigma$ yields
	\[
	e^{-\alpha_N M_N}Z_{N,0,\beta}
	\le
	Z_{N,\alpha_N,\beta}^\varepsilon
	\le
	e^{\alpha_N M_N}Z_{N,0,\beta}.
	\]
	Hence, for every event $E\subset\{-1,+1\}^N$,
	\[
	\mu_{N,\alpha_N,\beta}^\varepsilon(E)
	\le
	e^{2\alpha_N M_N}\mu_{N,0,\beta}(E).
	\]
	Applying this with $E=\{m^N\notin A_\kappa\}$ and using the moderate deviation bound
	\eqref{eq:MDP_decoupled}, we obtain
	\[
	\mu_{N,\alpha_N,\beta}^\varepsilon(m^N\notin A_\kappa)
	\le
	e^{2\alpha_N M_N}\mu_{N,0,\beta}(m^N\notin A_\kappa),
	\]
	which tends to $0$ whenever $\alpha_N M_N=o(N^{1-2\kappa})$.
	
	This proves the first assertion in probability, and the second one almost surely.
\end{proof}

	Now, fix a typical, admissible $(\mu_1,\mu_2)\in A_\kappa$ and condition on
	$m^N(\sigma)=(\mu_1,\mu_2)$. As in Lemma~\ref{lem:cross_conc}, under $\nu_N^{\mu_1,\mu_2}$, the \emph{uniform}
	measure on $\Omega_N(\mu_1,\mu_2)$, the total number of same-sign pairs across the
	(full) matching concentrates exponentially fast around its maximizer; equivalently,
	the empirical average of $\sigma_i\sigma_{i+N/2}$ across all $N/2$ matching edges is
	$\frac12\mu_1\mu_2+o(1)$ with probability $1-e^{-cN}$:
    \begin{lemma}
       Fix a typical and admissible $(\mu_1,\mu_2)\in A_\kappa$, then vor every $\eps>0$ there exists $\widetilde{c}_\eps >0$ such that
        $$
        \nu_N^{\mu_1,\mu_2}\left( \left\vert\frac{1}{M_N}\sum_{i\in B_1\, ; \, \eps_i = 1}\sigma_i\sigma_{i+\frac{N}{2}} - \mu_1\mu_2 \right\vert >\eps \right) \le e^{-\widetilde{c}_\eps M_N},
        $$
        for all $N$ large enough. 
    \end{lemma}
    \begin{proof}
   
	Restrict to the retained edge set
	$I:=\{i\in B_1:\varepsilon_i=1\}$ with $|I|=M_N$. 
	Recall that by Lemma~\ref{lem:cross_conc}, under the measure
	$\nu_N^{\mu_1,\mu_2}$, the total number of aligned edges across the full matching
	is
	\[
	\frac{N}{4}(1+\mu_1\mu_2)+o(N)
	\]
	with exponentially high probability.
	
	Conditioning on these total numbers of $++$ and $--$ edges across the full matching,
	the corresponding retained counts are hypergeometric ("drawing $M_N$-many times from $(\gamma^* + \gamma^{**})\frac{N}{4} = (1+\mu_1\mu_2) \frac{N}{4}$ many aligned pairs and $(1-\mu_1\mu_2)\frac{N}{4}$ many anti-aligned pairs"). By the Hoeffding--Serfling
	inequality for sampling without replacement, deviations from the mean
    $$
    \frac{2}{N}\left((+1)\cdot   (1+\mu_1\mu_2) \frac{N}{4} + (-1)\cdot(1-\mu_1\mu_2)\frac{N}{4} \right)= \mu_1\mu_2
    $$
    by $\eps$ are bounded by
	$\exp(-\widetilde{c}_\eps M_N)$, for some $\widetilde{c}_\eps >0$. 
	
	Combining this estimate with the exponential concentration from
	Lemma~\ref{lem:cross_conc}, we obtain
	\[
	\nu_N^{\mu_1,\mu_2}\left(
	\left|\frac{1}{M_N}\sum_{i\in I}\sigma_i\sigma_{i+N/2}-\mu_1\mu_2\right|>\varepsilon
	\right)
	\le e^{-cN}+e^{-\widetilde c_\varepsilon M_N}.
	\]
	Since $M_N\le N/2$, the right-hand side is bounded by
	$
	2e^{-c'_\varepsilon M_N}
	$
	for some $c'_\varepsilon>0$, and after adjusting the constant this yields
	\[
	\nu_N^{\mu_1,\mu_2}\left(
	\left|\frac{1}{M_N}\sum_{i\in I}\sigma_i\sigma_{i+N/2}-\mu_1\mu_2\right|>\varepsilon
	\right)
	\le e^{-c'_\varepsilon M_N}.
	\]
	
	Therefore, with conditional probability $1-e^{-\widetilde{c}_\eps M_N}$,
	\[
	\frac1{M_N}\sum_{i\in I}\sigma_i\sigma_{i+N/2}
	=\mu_1\mu_2+o(1).
	\]
	Thus, with $\nu_N^\mu$-probability at least $1-e^{- c'_\eps M_N}$,

	\begin{equation}\label{eq:quenched_cross_conc}
		\sum_{i\in B_1}\varepsilon_i\,\sigma_i\sigma_{i+N/2}
		=M_N\,\mu_1\mu_2+o(M_N).
	\end{equation}
	
\end{proof}

For an admissible magnetization vector $\mu=(\mu_1,\mu_2)$, let
\[
\Omega_N(\mu):=\{\sigma\in\{-1,+1\}^N:\ m_1(\sigma)=\mu_1,\ m_2(\sigma)=\mu_2\}
\]
and let $\nu_N^\mu$ denote the uniform probability measure on $\Omega_N(\mu)$.
Moreover, define
\[
S_\varepsilon(\sigma):=\sum_{i\in B_1}\varepsilon_i\,\sigma_i\sigma_{i+N/2}.
\]
Since the Curie--Weiss part of the Hamiltonian is constant on $\Omega_N(\mu)$, we have
\begin{equation}\label{eq:mu-decomp-quenched}
	\mu_{N,\alpha_N,\beta}^\varepsilon(m^N=\mu)
	=
	\frac{1}{Z_{N,\alpha_N,\beta}^\varepsilon}
	\exp\!\left(\frac{\beta N}{8}(\mu_1^2+\mu_2^2)\right)
	|\Omega_N(\mu)|
	\mathbb E_{\nu_N^\mu}\!\left[e^{\alpha_N S_\varepsilon}\right].
\end{equation}

\medskip
\begin{corollary}\label{cor:tilted_expectation_quenched}
	Let $\mu=(\mu_1,\mu_2)\in A_\kappa$ be typical and admissible, and fix $\delta>0$.
	Then
	\[
	\exp\!\left(\alpha_N M_N\mu_1\mu_2-\delta \alpha_N M_N\right)
	\bigl(1-e^{-c_\delta M_N}\bigr)
	\le
	\mathbb E_{\nu_N^\mu}\!\left[e^{\alpha_N S_\varepsilon}\right]
	\]
	and
	\[
	\mathbb E_{\nu_N^\mu}\!\left[e^{\alpha_N S_\varepsilon}\right]
	\le
	\exp\!\left(\alpha_N M_N\mu_1\mu_2+\delta \alpha_N M_N\right)
	+
	e^{\alpha_N M_N-c_\delta M_N}.
	\]
	In particular, if $\alpha_N\to0$ and $M_N\to\infty$, then
	\be \label{eq:tilted_expectation_asymp_quenched}
	\log \mathbb E_{\nu_N^\mu}\!\left[e^{\alpha_N S_\varepsilon}\right]
	=
	\alpha_N M_N\mu_1\mu_2+o(\alpha_N M_N).
	\ee
\end{corollary}

\begin{proof}
	Fix $\delta>0$ and define
	\[
	G_{\mu,\delta}^\varepsilon
	:=
	\left\{
	\sigma\in\Omega_N(\mu):
	\left|\frac{1}{M_N}S_\varepsilon(\sigma)-\mu_1\mu_2\right|\le \delta
	\right\}.
	\]
	By the previous lemma, 
	\[
	\nu_N^\mu\bigl((G_{\mu,\delta}^\varepsilon)^c\bigr)\le e^{-c_\delta M_N}
	\]
	for a constant $c_\delta>0$ and for all $N$ sufficiently large.
	
	On $G_{\mu,\delta}^\varepsilon$ we have
	\[
	\exp\!\left(\alpha_N M_N(\mu_1\mu_2-\delta)\right)
	\le e^{\alpha_N S_\varepsilon}
	\le
	\exp\!\left(\alpha_N M_N(\mu_1\mu_2+\delta)\right).
	\]
	Therefore,
	\begin{align*}
		\mathbb E_{\nu_N^\mu}\!\left[e^{\alpha_N S_\varepsilon}\right]
		&\ge
		\mathbb E_{\nu_N^\mu}\!\left[e^{\alpha_N S_\varepsilon}\mathbf 1_{G_{\mu,\delta}^\varepsilon}\right] \\
		&\ge
		\exp\!\left(\alpha_N M_N(\mu_1\mu_2-\delta)\right)
		\nu_N^\mu(G_{\mu,\delta}^\varepsilon),
	\end{align*}
	which proves the lower bound.
	
	Similarly,
	\begin{align*}
		\mathbb E_{\nu_N^\mu}\!\left[e^{\alpha_N S_\varepsilon}\right]
		&\le
		\mathbb E_{\nu_N^\mu}\!\left[e^{\alpha_N S_\varepsilon}\mathbf 1_{G_{\mu,\delta}^\varepsilon}\right]
		+
		\mathbb E_{\nu_N^\mu}\!\left[e^{\alpha_N S_\varepsilon}\mathbf 1_{(G_{\mu,\delta}^\varepsilon)^c}\right] \\
		&\le
		\exp\!\left(\alpha_N M_N(\mu_1\mu_2+\delta)\right)
		+
		e^{\alpha_N M_N}\nu_N^\mu\bigl((G_{\mu,\delta}^\varepsilon)^c\bigr),
	\end{align*}
	which yields the upper bound. Since $\alpha_N\to0$, the second term is
	exponentially negligible relative to the scale $\alpha_N M_N$, and
	\eqref{eq:tilted_expectation_asymp_quenched} follows.
\end{proof}

\medskip

\begin{proof}[Proof of Theorem~\ref{theo:theo2}]
	With $\varepsilon$-probability tending to one or $\P_\varepsilon$-almost surely (depending on which condition in Theorem \ref{theo:theo2} we assume), Proposition~\ref{prop:atypical_mu_thinned}
	implies
	\[
	\mu_{N,\alpha_N,\beta}^\varepsilon\bigl(m^N\notin A_\kappa\bigr)\longrightarrow 0,
	\]
	so it suffices to analyze the Gibbs weights of admissible magnetization vectors
	$\mu=(\mu_1,\mu_2)\in A_\kappa$.
	
	Fix such a $\mu$. By \eqref{eq:mu-decomp-quenched} and
	Corollary~\ref{cor:tilted_expectation_quenched} (since by \eqref{eq:MN_good_event} $M_N\to \infty$ with $\eps$-probability tending to one or $\P_\eps$-almost surely),
	\begin{equation}\label{eq:mu_weight_asymp_quenched}
		\mu_{N,\alpha_N,\beta}^\varepsilon(m^N=\mu)
		=
		\frac{|\Omega_N(\mu)|}{Z_{N,\alpha_N,\beta}^\varepsilon}
		\exp\!\left(
		\frac{\beta N}{8}(\mu_1^2+\mu_2^2)
		+
		\alpha_N M_N\mu_1\mu_2
		+
		o(\alpha_N M_N)
		\right).
	\end{equation}

	Now let $\mu=(\mu_1,\mu_2)\in A_\kappa$ and $\tilde\mu=(\mu_1,-\mu_2)\in A_\kappa$.
	As in the proof of Theorem~\ref{theo:theo1}, we have
	\[
	|\Omega_N(\mu)|=|\Omega_N(\tilde\mu)|
	\qquad\text{and}\qquad
	\mu_1^2+\mu_2^2=\mu_1^2+(-\mu_2)^2.
	\]
	Hence \eqref{eq:mu_weight_asymp_quenched} yields
	\[
	\frac{\mu_{N,\alpha_N,\beta}^\varepsilon(m^N=\mu)}
	{\mu_{N,\alpha_N,\beta}^\varepsilon(m^N=\tilde\mu)}
	=
	\exp\!\left(2\alpha_N M_N\mu_1\mu_2+o(\alpha_N M_N)\right).
	\]
	
	For $\mu\in A_\kappa$, we have
	\[
	\mu_i=\pm m^*(\beta/2)+O(N^{-\kappa}), \qquad i=1,2.
	\]
	Therefore, the ratio between an aligned and an anti-aligned admissible
	magnetization in $A_\kappa$ is of the form
	\[
	\exp\!\left(2\alpha_N M_N (m^*(\beta/2))^2 + o(\alpha_N M_N)\right).
	\]
	
	By spin-flip symmetry, the two aligned wells have equal total quenched Gibbs
	mass, and the two anti-aligned wells have equal total quenched Gibbs mass.
	Since the number of admissible magnetization values in each well grows at most
	polynomially in $N$, this polynomial multiplicity does not affect the above
	exponential comparison.
	
	Finally, by \eqref{eq:MN_good_event}, with $\varepsilon$-probability tending to one, or
	$\P_\eps$-almost surely (depending on which condition in Theorem \ref{theo:theo2} we assume) 
     
	\begin{equation}\label{eq:MN_interval}
	    M_N \in \left[ \frac N2 p(N)(1-\delta),  \frac N2 p(N)(1+\delta) \right].
	\end{equation}
	Hence the decisive scale is $\alpha_N p(N)N$.
	
	If $\alpha_N p(N)N\to\infty$, then the aligned wells dominate and the quenched
	distribution of $m^N$ converges to
	\[
	\frac12\Bigl(
	\delta_{(m^*(\beta/2),m^*(\beta/2))}
	+
	\delta_{(-m^*(\beta/2),-m^*(\beta/2))}
	\Bigr).
	\]
	If $\alpha_N p(N)N\to0$, then the ratio tends to $1$, and by symmetry all four
	wells carry equal asymptotic mass. Hence the quenched distribution of $m^N$
	converges to
	\[
	\frac14\Bigl(
	\delta_{(m^*(\beta/2),m^*(\beta/2))}
	+\delta_{(m^*(\beta/2),-m^*(\beta/2))}
	+\delta_{(-m^*(\beta/2),m^*(\beta/2))}
	+\delta_{(-m^*(\beta/2),-m^*(\beta/2))}
	\Bigr)
	\]
	with $\varepsilon$-probability tending to one or $\P_\varepsilon$-almost surely depending on which condition in Theorem \ref{theo:theo2} we assume. This proves convergence in distribution (in probability or almost surely). 

    \medskip

\end{proof}

\section{Proof of Theorems \ref{thm:Blockspin} and \ref{thm:PhaseTransition}}

Let us start by recalling the definition of the Hamiltonian in the three-block model
$$
H_{N,\beta,\alpha_N}(m) = - \sum_{k=1}^3 |B_k|\frac{\beta}{2}m_k^2 - \alpha_N\sqrt{b_N\frac{N-b_N}{2}}(m_1+m_3)m_2
$$
and note that we can write $H_{N,\beta,\alpha_N}$ as the sum of three Hamiltonian of three standard Curie-Weiss models (on the respective blocks) plus the inter-block interaction. More precisely,
$$
H_{N,\beta,\alpha_N}(m) = \overline{H}_{\frac{N-b_N}{2}, \beta}(m_1) + \overline{H}_{b_N, \beta}(m_2) + \overline{H}_{\frac{N-b_N}{2},\beta}(m_3) - \alpha_N\sqrt{b_N \frac{N-b_N}{2}}m_2(m_1+m_3),
$$
where we denote by $\overline{H}_{n , \beta}$ the Hamiltonian of a standard Curie-Weiss model of size $n$ at inverse temperature $\beta$. 
We will start by showing that the magnetizations of the first and of the third block converge to $\pm m^*$ in all of the regimes (i.e.\ regardless of the limit behavior of $m_2$). In the following for any $\eps >0$, we will use the abbreviation
\begin{equation}\label{def:Aeps}
	A_\eps = [m^*-\eps, m^* + \eps]\cup [-m^*-\eps, -m^*+ \eps] .
\end{equation}
The first observation is:
\begin{lemma}\label{lemma:m1}
	For any $0<\eps< \min\{m^*, 1-m^*\}$
	$$
	\mu_{N,\beta, \alpha_N}\left( m_1 \notin A_\eps \right) \le \exp\Bigl\{-n_N\bigl(c_\beta \varepsilon^2+o(1)\bigr)\Bigr\} \longrightarrow 0,
	$$
	as $N\to \infty$, for a constant $c_\beta >0$. In other words, the magnetization of the first block converges in probability under the Gibbs measure to $\pm m^*$. In particular, by exchangeability, the same is true also for $m_3$ and by a union bound we obtain
	\begin{flalign*}
		\mu_{N,\beta, \alpha_N}\left( m_1,m_3 \in A_\eps \right) \ge1-  2\mu_{N,\beta, \alpha_N}\left( m_1 \notin A_\eps\right)   \longrightarrow 1.
	\end{flalign*}
	
\end{lemma}

\begin{proof}
	Let
	\[
	n_N:=|B_1|=|B_3|=\frac{N-b_N}{2},
	\qquad |B_2|=b_N.
	\]
	The block magnetizations take values in
	\[
	\mathcal A_N^{(k)}=\left\{-1+\frac{2\ell}{|B_k|}:\ \ell=0,\dots,|B_k|\right\},
	\qquad k=1,2,3.
	\]
	
	Recall that the Gibbs measure can be written in terms of the magnetization vector
	$m=(m_1,m_2,m_3)$, and that
	\[
	Z_{N,\beta,\alpha_N}
	=
	\sum_{m_1,m_2,m_3}
	\exp\{-H_{N,\beta,\alpha_N}(m)\}
	\frac{1}{2^N}
	\prod_{k=1}^3 \binom{|B_k|}{|B_k|\frac{1+m_k}{2}}.
	\]
	Using Stirling's formula, we may rewrite this as
	\begin{multline*}
	Z_{N,\beta,\alpha_N}
	=\\
	\sum_{m_1,m_2,m_3}
	\exp\Bigl\{
	n_N\bigl(F_\beta(m_1)+F_\beta(m_3)\bigr)
	+b_NF_\beta(m_2)
	+\alpha_N\sqrt{b_Nn_N}\,m_2(m_1+m_3)
	-\sum_{k=1}^3 R_N^{(k)}(m_k)
	\Bigr\},
	\end{multline*}
	where
	\[
	F_\beta(x):=\frac{\beta}{2}x^2-\log 2+s(x),
	\]
	\[
	s(x):=-\frac{1+x}{2}\log \frac{1+x}{2}
	-\frac{1-x}{2}\log \frac{1-x}{2},
	\]
	and
	\[
	R_N^{(k)}(m)
	=
	\frac12\log\!\Bigl(|B_k|(1-m^2)\frac{\pi}{2}\Bigr)
	+O\!\left(\frac1{|B_k|}\right)
	=
	o(|B_k|).
	\]
	
	Let $\widetilde m_N$ be a maximizer of $F_\beta$ on $\mathcal A_N^{(1)}$, and
	let $\widetilde m_{b_N}$ be a maximizer of $F_\beta$ on $\mathcal A_N^{(2)}$.
	Since $F_\beta$ has exactly two maximizers $\pm m^*$ on $[-1,1]$, we have
	\[
	\widetilde m_N\to \pm m^*
	\qquad \text{and }\quad
	\widetilde m_{b_N}\to \pm m^*
	\]
	(i.e.\ up to the choice of sign); in particular,
	\[
	F_\beta(\widetilde m_N)\to F_\beta(m^*),
	\qquad
	F_\beta(\widetilde m_{b_N})\to F_\beta(m^*).
	\]
	
	We first derive a lower bound on the partition function by keeping only the
	single summand $(m_1,m_2,m_3)=(\widetilde m_N,\widetilde m_{b_N},\widetilde m_N)$:
	\begin{equation}\label{eq:Zlower_m1}
		Z_{N,\beta,\alpha_N}
		\ge
		\exp\Bigl\{
		2n_NF_\beta(\widetilde m_N)
		+b_NF_\beta(\widetilde m_{b_N})
		+2\alpha_N\sqrt{b_Nn_N}\,\widetilde m_N\widetilde m_{b_N}
		-2R_N^{(1)}(\widetilde m_N)-R_N^{(2)}(\widetilde m_{b_N})
		\Bigr\}.
	\end{equation}
	
	Next, we estimate the numerator of
	$\mu_{N,\beta,\alpha_N}(m_1\notin A_\varepsilon)$.
	For all $m_1,m_2,m_3\in[-1,1]$,
	\[
	m_2(m_1+m_3)\le 2,
	\]
	hence
	\[
	\alpha_N\sqrt{b_Nn_N}\,m_2(m_1+m_3)
	-2\alpha_N\sqrt{b_Nn_N}\,\widetilde m_N\widetilde m_{b_N}
	\le
	2\alpha_N\sqrt{b_Nn_N}\bigl(1-\widetilde m_N\widetilde m_{b_N}\bigr)
	\le
	2\alpha_N\sqrt{b_Nn_N}.
	\]
	Therefore, using \eqref{eq:Zlower_m1},
	\begin{align*}
		\mu_{N,\beta,\alpha_N}(m_1\notin A_\varepsilon)
		&\le
		\sum_{m_1\notin A_\varepsilon}\sum_{m_2,m_3}
		\exp\Bigl\{
		n_N\bigl(F_\beta(m_1)-F_\beta(\widetilde m_N)\bigr)
		+n_N\bigl(F_\beta(m_3)-F_\beta(\widetilde m_N)\bigr) \\
		&\hspace{5em}
		+b_N\bigl(F_\beta(m_2)-F_\beta(\widetilde m_{b_N})\bigr)
		+2\alpha_N\sqrt{b_Nn_N}
		+2R_N^{(1)}(\widetilde m_N)+R_N^{(2)}(\widetilde m_{b_N})
		\Bigr\}.
	\end{align*}
	Since $\widetilde m_N$ and $\widetilde m_{b_N}$ maximize $F_\beta$ on the
	corresponding discrete sets we obtain,
	\[
	F_\beta(m_3)-F_\beta(\widetilde m_N)\le 0
	\qquad \text{as well as }\quad 
	F_\beta(m_2)-F_\beta(\widetilde m_{b_N})\le 0.
	\]
	Thus
	\begin{align}
		\mu_{N,\beta,\alpha_N}(m_1\notin A_\varepsilon)
		&\le
		|\mathcal A_N^{(2)}|\,|\mathcal A_N^{(3)}|
		\exp\Bigl\{
		2\alpha_N\sqrt{b_Nn_N}
		+2R_N^{(1)}(\widetilde m_N)+R_N^{(2)}(\widetilde m_{b_N})
		\Bigr\}
		\nonumber\\
		&\qquad\qquad\times
		\sum_{m_1\notin A_\varepsilon}
		\exp\Bigl\{
		n_N\bigl(F_\beta(m_1)-F_\beta(\widetilde m_N)\bigr)
		\Bigr\}.
		\label{eq:m1_reduction}
	\end{align}
    It remains to estimate the last sum. By a Taylor expansion of $F_\beta$ around its minimizers we obtain that there exists a constant $c_\beta >0$ such that for every
	$0<\varepsilon<\min\{m^*,1-m^*\}$
	\[
	\sup_{x\notin A_\varepsilon} F_\beta(x)
	\le
	F_\beta(m^*)-c_\beta \varepsilon^2.
	\]
	Because $F_\beta(\widetilde m_N)\to F_\beta(m^*)$, it follows that uniformly in
	$m_1\notin A_\varepsilon$,
	\[
	F_\beta(m_1)-F_\beta(\widetilde m_N)
	\le
	-c_\beta \varepsilon^2+o(1).
	\]
	Hence
	\[
	\sum_{m_1\notin A_\varepsilon}
	\exp\Bigl\{
	n_N\bigl(F_\beta(m_1)-F_\beta(\widetilde m_N)\bigr)
	\Bigr\}
	\le
	|\mathcal A_N^{(1)}|
	\exp\Bigl\{-n_N\bigl(c_\beta \varepsilon^2+o(1)\bigr)\Bigr\}.
	\]
	Inserting this into \eqref{eq:m1_reduction} yields
	\begin{align*}
		\mu_{N,\beta,\alpha_N}(m_1\notin A_\varepsilon)
		&\le
		|\mathcal A_N^{(1)}|\,|\mathcal A_N^{(2)}|\,|\mathcal A_N^{(3)}|
		\exp\Bigl\{
		2\alpha_N\sqrt{b_Nn_N}
		+2R_N^{(1)}(\widetilde m_N)+R_N^{(2)}(\widetilde m_{b_N})
		\Bigr\}
		\\
		&\qquad\qquad\times
		\exp\Bigl\{-n_N\bigl(c_\beta \varepsilon^2+o(1)\bigr)\Bigr\}.
	\end{align*}
	Now
	\[
	|\mathcal A_N^{(1)}|=|\mathcal A_N^{(3)}|=n_N+1,
	\qquad
	|\mathcal A_N^{(2)}|=b_N+1,
	\]
	and
	\[
	R_N^{(1)}(\widetilde m_N)=O(\log n_N)
	\qquad \text{and}\quad 
	R_N^{(2)}(\widetilde m_{b_N})=O(\log b_N).
	\]
	Therefore the prefactor is at most polynomial in $N$ and $b_N$, whereas
	\[
	2\alpha_N\sqrt{b_Nn_N}=o(n_N),
	\]
	because $n_N\sim N/2$, $b_N=o(N)$, and $\alpha_N\to0$.
	Hence
	\[
	\mu_{N,\beta,\alpha_N}(m_1\notin A_\varepsilon)
	\le
	\operatorname{poly}(N,b_N)\,
	\exp\Bigl\{-n_N\bigl(c_\beta \varepsilon^2+o(1)\bigr)\Bigr\}
	\longrightarrow 0.
	\]
	
	This proves the first claim. The corresponding statement for $m_3$ follows by
	symmetry, and the joint claim follows by the union bound.
\end{proof}

By Lemma \ref{lemma:m1}, we know that the magnetizations $m_1$ and $m_3$ converge in probability to $\pm m^*$. In the next lemma, we show that if $\lim_{N\to \infty} \alpha_N \sqrt{b_N N} = \infty$, then the interaction is strong enough such that in the limit $m_1$ and $m_3$ have the same sign.

\begin{lemma}\label{lemma:m1m3Regime1and2}
	If $\lim_{N\to \infty} \alpha_N \sqrt{b_N N} = \infty$, then for any $0<\eps< \frac{(m^*)^2}{1+m^*}$,
	$$
	\frac{\mu_{N,\beta, \alpha_N}\left( m_1 \in [m^*-\eps, m^*+\eps], m_3 \in [-m^*-\eps, -m^*+\eps]  \right)}{\mu_{N,\beta, \alpha_N}\left( m_1,m_3 \in [m^*-\eps, m^*+\eps]\right)} \longrightarrow 0.
	$$
	In particular, since the Gibbs measure is invariant under a global spin flip, it follows that
	\begin{flalign*}
		&\frac{\mu_{N,\beta, \alpha_N}\left( m_1, m_3 \in A_\eps, \, m_1m_3 <0 \right)}{\mu_{N,\beta, \alpha_N}\left( m_1,m_3 \in A_\eps, \, m_1m_3 >0\right)} \\
		\le & 4  \frac{\mu_{N,\beta, \alpha_N}\left( m_1 \in [m^*-\eps, m^*+\eps], m_3 \in [-m^*-\eps, -m^*+\eps]  \right)}{\mu_{N,\beta, \alpha_N}\left( m_1,m_3 \in [m^*-\eps, m^*+\eps]\right)}  \longrightarrow 0.
	\end{flalign*}
	
\end{lemma}
\begin{proof}
	First note that
	\[
	\mu_{N,\beta,\alpha_N}\bigl(m_1,m_3\in[m^*-\eps,m^*+\eps]\bigr)
	\ge
	\mu_{N,\beta,\alpha_N}\bigl(m_1,m_3\in[m^*-\eps,m^*+\eps],\ m_2\ge m^*\bigr),
	\]
	and therefore
	\begin{multline*}
		\frac{\mu_{N,\beta,\alpha_N}\bigl(m_1\in[m^*-\eps,m^*+\eps],\ 
			m_3\in[-m^*-\eps,-m^*+\eps]\bigr)}
		{\mu_{N,\beta,\alpha_N}\bigl(m_1,m_3\in[m^*-\eps,m^*+\eps]\bigr)} \\
		\le
		\frac{\mu_{N,\beta,\alpha_N}\bigl(m_1\in[m^*-\eps,m^*+\eps],\ 
			m_3\in[-m^*-\eps,-m^*+\eps]\bigr)}
		{\mu_{N,\beta,\alpha_N}\bigl(m_1,m_3\in[m^*-\eps,m^*+\eps],\ m_2\ge m^*\bigr)}.
	\end{multline*}
	
	On the set
	\[
	\{m_1\in[m^*-\eps,m^*+\eps],\ m_3\in[-m^*-\eps,-m^*+\eps]\}
	\]
	we have
	\[
	|m_1+m_3|\le 2\eps.
	\]
	Hence
	\begin{align}
		&\mu_{N,\beta,\alpha_N}\bigl(m_1\in[m^*-\eps,m^*+\eps],\ 
		m_3\in[-m^*-\eps,-m^*+\eps]\bigr)
		\nonumber\\
		&\le
		\frac{1}{Z_{N,\beta,\alpha_N}}
		\exp\left\{
		2\alpha_N\sqrt{b_N\frac{N-b_N}{2}}\,\eps
		\right\}
		\sum_{m_1,m_2,m_3}
		\mathbf 1_{\{m_1\in[m^*-\eps,m^*+\eps]\}}
		\mathbf 1_{\{m_3\in[-m^*-\eps,-m^*+\eps]\}}
		\nonumber\\
		&\qquad\qquad\times
		\exp\left\{-
		\overline H_{\frac{N-b_N}{2},\beta}(m_1)
		-\overline H_{b_N,\beta}(m_2)
		-\overline H_{\frac{N-b_N}{2},\beta}(m_3)
		\right\}
		\prod_{k=1}^3 \binom{|B_k|}{|B_k|\frac{1+m_k}{2}}.
		\label{eq:num_regime12}
	\end{align}
	
	On the other hand, on the set
	$
	\{m_1,m_3\in[m^*-\eps,m^*+\eps],\ m_2\ge m^*\}
	$
	we have
	$
	m_1+m_3\ge 2(m^*-\eps)$, and
	$
	m_2\ge m^*,
	$
	and therefore
	\[
	m_2(m_1+m_3)\ge 2m^*(m^*-\eps).
	\]
	Thus
	\begin{align}
		&\mu_{N,\beta,\alpha_N}\bigl(m_1,m_3\in[m^*-\eps,m^*+\eps],\ m_2\ge m^*\bigr)
		\nonumber\\
		&\ge
		\frac{1}{Z_{N,\beta,\alpha_N}}
		\exp\left\{
		2\alpha_N\sqrt{b_N\frac{N-b_N}{2}}\,m^*(m^*-\eps)
		\right\}
		\sum_{m_1,m_2,m_3}
		\mathbf 1_{\{m_1,m_3\in[m^*-\eps,m^*+\eps]\}}
		\mathbf 1_{\{m_2\ge m^*\}}
		\nonumber\\
		&\qquad\qquad\times
		\exp\left\{ -
		\overline H_{\frac{N-b_N}{2},\beta}(m_1)
		-\overline H_{b_N,\beta}(m_2)
		-\overline H_{\frac{N-b_N}{2},\beta}(m_3)
		\right\}
		\prod_{k=1}^3 \binom{|B_k|}{|B_k|\frac{1+m_k}{2}}.
		\label{eq:den_regime12}
	\end{align}
	
	Now observe that the decoupled Curie--Weiss weights are symmetric under $m\mapsto -m$. Hence the sum over
	$m_3\in[-m^*-\eps,-m^*+\eps]$ in \eqref{eq:num_regime12} equals the
	corresponding sum over $m_3\in[m^*-\eps,m^*+\eps]$, and similarly for $m_1$.
	Therefore the $m_1$- and $m_3$-sums cancel in the ratio
	\eqref{eq:num_regime12}/\eqref{eq:den_regime12}. For the $m_2$-sum we obtain
	\[
	\sum_{m_2}\exp\{ -\overline H_{b_N,\beta}(m_2)\}
	\binom{b_N}{b_N\frac{1+m_2}{2}}
	=
	Z_{b_N,\beta},
	\]
	whereas the restricted sum in the denominator is
	$
	Z_{b_N,\beta}\,\mu_{b_N,\beta}(m_2\ge m^*).
	$
	Consequently,
	\begin{multline*}
		\frac{\mu_{N,\beta,\alpha_N}\bigl(m_1\in[m^*-\eps,m^*+\eps],\
			m_3\in[-m^*-\eps,-m^*+\eps]\bigr)}
		{\mu_{N,\beta,\alpha_N}\bigl(m_1,m_3\in[m^*-\eps,m^*+\eps],\ m_2\ge m^*\bigr)}
		\\
		\le
		\frac{1}{\mu_{b_N,\beta}(m_2\ge m^*)}
		\exp\left\{
		2\alpha_N\sqrt{b_N\frac{N-b_N}{2}}
		\bigl(\eps-m^*(m^*-\eps)\bigr)
		\right\}.
	\end{multline*}
	
	Since $\beta>1$, the Curie--Weiss magnetization under $\mu_{b_N,\beta}$ converges
	to $\pm m^*$ with equal asymptotic weights, and therefore
	\[
	\mu_{b_N,\beta}(m_2\ge m^*)\longrightarrow \frac12.
	\]
	Moreover,
	$
	\eps-m^*(m^*-\eps)=\eps(1+m^*)-(m^*)^2<0
	$
	by the assumption
	$
	0<\eps<\frac{(m^*)^2}{1+m^*}.
	$
	Since
	\[
	\alpha_N\sqrt{b_NN}\to\infty
	\qquad\text{and}\qquad
	\frac{N-b_N}{N}\to1,
	\]
	the exponential term tends to $0$, and the claim follows.
	
	The final statement is then immediate from global spin-flip symmetry.
\end{proof}

Now that we have established the convergence of $m_1$ and $m_3$ to $m^*$ or to $-m^*$ in the case where $\lim_{N\to \infty}\alpha_N \sqrt{b_NN} = \infty$, regardless of the behavior of $m_2$. This means that effectively, the second block views its larger neighbors as a (random and converging) external magnetic field of strength $\alpha_N\sqrt{b_N\frac{N-b_N}{2}}(m_1+m_3)$. In the following three lemmas, we will show that in the limit it indeed displays the behaviour of a Curie-Weiss model with external magnetic field of strength $\lim_{N\to \infty}\alpha_N\sqrt{b_N\frac{N-b_N}{2}}(m_1+m_3)$. In Lemma \ref{lem:c=infty} (for the case $c=\infty$) and in Lemma \ref{lemma:c<infty} (for the case $c \in [0,\infty)$), we will show that if $m_2$ has the same sign as $m_1$ and $m_3$, then it will converge to $\pm m(c)$ (with the same sign). Afterwards, in Lemma \ref{lem:samesign} we will show that $m_2$ indeed has the same sign as $m_1$ and $m_3$ in the limit.

\begin{lemma}\label{lem:c=infty}
	If $\lim_{N\to \infty}\alpha_N\sqrt{\frac{N}{b_N}}=\infty$ (i.e.\ $c= \infty$), then for any $0< \delta < 1$ and $0<\eps< \min\{m^*, 1-m^*, \frac{m^*}{2}\delta\}$,
	$$
	\frac{\mu_{N,\beta,\alpha_N} \left( m_1,m_3 \in [m^*-\eps,m^*+\eps], m_2< 1-\delta \right)}{\mu_{N,\beta,\alpha_N} \left( m_1,m_3 \in [m^*-\eps,m^*+\eps]\right)} \longrightarrow 0
	$$
	and again by invariance under a global spin flip, this implies that also
	$$
	\frac{\mu_{N,\beta,\alpha_N} \left( m_1,m_3 \in [-m^*-\eps,-m^*+\eps], m_2>- 1+\delta \right)}{\mu_{N,\beta,\alpha_N} \left( m_1,m_3 \in [-m^*-\eps,-m^*+\eps]\right)} \longrightarrow 0.
	$$
\end{lemma}

\begin{proof}
	Let $
	I_\eps:=[m^*-\eps,m^*+\eps].$
	As before,
	\[
	\mu_{N,\beta,\alpha_N}(m_1,m_3\in I_\eps)
	\ge
	\mu_{N,\beta,\alpha_N}(m_1,m_3\in I_\eps,\ m_2=1),
	\]
	and therefore
	\begin{align*}
		&\frac{\mu_{N,\beta,\alpha_N}(m_1,m_3\in I_\eps,\ m_2<1-\delta)}
		{\mu_{N,\beta,\alpha_N}(m_1,m_3\in I_\eps)}
		\le
		\frac{\mu_{N,\beta,\alpha_N}(m_1,m_3\in I_\eps,\ m_2<1-\delta)}
		{\mu_{N,\beta,\alpha_N}(m_1,m_3\in I_\eps,\ m_2=1)}.
	\end{align*}
	Again we write $
	n_N:=\frac{N-b_N}{2}.$
	On the set \(\{m_1,m_3\in I_\eps,\ m_2<1-\delta\}\), we have
	$
	m_1+m_3\le 2(m^*+\eps)$ and
	$m_2\le 1-\delta$,
	hence
	$
	m_2(m_1+m_3)\le 2(1-\delta)(m^*+\eps)$.
	Therefore,
	\begin{align}
		&\mu_{N,\beta,\alpha_N}(m_1,m_3\in I_\eps,\ m_2<1-\delta)
		\nonumber\\
		&\le
		\frac{1}{Z_{N,\beta,\alpha_N}}
		\exp\left\{
		2\alpha_N\sqrt{b_Nn_N}(1-\delta)(m^*+\eps)
		\right\}
		\sum_{m_1,m_2,m_3}
		\mathbf 1_{\{m_1,m_3\in I_\eps\}}
		\mathbf 1_{\{m_2<1-\delta\}}
		\nonumber\\
		&\qquad\qquad\times
		\exp\left\{-
		\overline H_{n_N,\beta}(m_1)
		-\overline H_{b_N,\beta}(m_2)
		-\overline H_{n_N,\beta}(m_3)
		\right\}
		\prod_{k=1}^3 \binom{|B_k|}{|B_k|\frac{1+m_k}{2}}.
		\label{eq:num_cinf}
	\end{align}
	
	On the other hand, on the set \(\{m_1,m_3\in I_\eps,\ m_2=1\}\), we have
	$m_1+m_3\ge 2(m^*-\eps)$,
	so
	\begin{align}
		&\mu_{N,\beta,\alpha_N}(m_1,m_3\in I_\eps,\ m_2=1)
		\nonumber\\
		&\ge
		\frac{1}{Z_{N,\beta,\alpha_N}}
		\exp\left\{
		2\alpha_N\sqrt{b_Nn_N}(m^*-\eps)
		\right\}
		\sum_{m_1,m_3}
		\mathbf 1_{\{m_1,m_3\in I_\eps\}}
		\nonumber\\
		&\qquad\qquad\times
		\exp\left\{-
		\overline H_{n_N,\beta}(m_1)
		-\overline H_{b_N,\beta}(1)
		-\overline H_{n_N,\beta}(m_3)
		\right\}
		\prod_{k=1}^3 \binom{|B_k|}{|B_k|\frac{1+m_k}{2}}.
		\label{eq:den_cinf}
	\end{align}
	
	Taking the ratio of \eqref{eq:num_cinf} and \eqref{eq:den_cinf}, the sums over
	\(m_1\) and \(m_3\) cancel, and we obtain
	\begin{align}
		&\frac{\mu_{N,\beta,\alpha_N}(m_1,m_3\in I_\eps,\ m_2<1-\delta)}
		{\mu_{N,\beta,\alpha_N}(m_1,m_3\in I_\eps,\ m_2=1)}
		\nonumber\\
		&\le
		\exp\left\{
		2\alpha_N\sqrt{b_Nn_N}\bigl((1-\delta)(m^*+\eps)-(m^*-\eps)\bigr)
		\right\} \\
        & \qquad \times 
		\sum_{m_2<1-\delta}
		\exp\left\{-
		\overline H_{b_N,\beta}(m_2)+\overline H_{b_N,\beta}(1)
		\right\}
		\binom{b_N}{b_N\frac{1+m_2}{2}}.
		\label{eq:ratio_cinf_1}
	\end{align}
	
	Now let \(\widetilde m_{b_N}\) be the positive maximizer of \(F_\beta\) on the
	discrete set \(\mathcal A_N^{(2)}\). Then for all \(m_2\in\mathcal A_N^{(2)}\),
	$
	F_\beta(m_2)\le F_\beta(\widetilde m_{b_N}),
	$
	and therefore
	\begin{multline*}
	    \sum_{m_2<1-\delta}
	\exp\left\{-
	\overline H_{b_N,\beta}(m_2)+\overline H_{b_N,\beta}(1)
	\right\}
	\binom{b_N}{b_N\frac{1+m_2}{2}} \\
	\le
	(b_N+1)\exp\left\{
	b_N\bigl(F_\beta(\widetilde m_{b_N})-F_\beta(1)\bigr)+O(\log b_N)
	\right\}.
	\end{multline*}
	Since \(\widetilde m_{b_N}\to m^*\), we have
	$F_\beta(\widetilde m_{b_N})-F_\beta(1)
	=F_\beta(m^*)-F_\beta(1)+o(1)$.
	Hence \eqref{eq:ratio_cinf_1} gives
	\begin{align*}
		&\frac{\mu_{N,\beta,\alpha_N}(m_1,m_3\in I_\eps,\ m_2<1-\delta)}
		{\mu_{N,\beta,\alpha_N}(m_1,m_3\in I_\eps,\ m_2=1)}
		\\
		&\le
		\exp\Bigl\{
		2\alpha_N\sqrt{b_Nn_N}\bigl(2\eps-\delta(m^*+\eps)\bigr)
		+
		b_N\bigl(F_\beta(m^*)-F_\beta(1)+o(1)\bigr)
		+O(\log b_N)
		\Bigr\}.
	\end{align*}
	
	By assumption we know, $\frac{2\eps}{m^*}<\delta<1$.
	In particular,
	$2\eps-\delta(m^*+\eps)<0.$
	Moreover, $	\alpha_N\sqrt{\frac{N}{b_N}}\to\infty$ implies
	$\alpha_N\sqrt{Nb_N}\gg b_N,$
	and since \(n_N\sim N/2\), the negative term of order
	\(\alpha_N\sqrt{Nb_N}\) dominates the positive term of order \(b_N\).
	Therefore the right-hand side tends to \(0\), proving the first claim.
	
	The second claim follows by global spin-flip symmetry.
\end{proof}

\begin{lemma}\label{lemma:c<infty}
	If $\lim_{N\to \infty} \alpha_N \sqrt{\frac{N}{b_N}}=c \in [0,\infty)$, then for any $0<\delta <m(c)$, and $0<\eps<\eps_{\delta,c,\beta}$ (for some $\eps_{\delta,c,\beta}>0$ and to be chosen later),
	$$
	\frac{\mu_{N,\beta,\alpha_N} \left( m_1,m_3 \in [m^*-\eps, m^*+\eps], m_2\in [m(c)-\delta, m(c)+\delta]^c\cap[0,1] \right)}{\mu_{N,\beta,\alpha_N} \left( m_1,m_3 \in  [m^*-\eps, m^*+\eps]\right)} \to 0.
	$$
\end{lemma}

\begin{proof}
Again let $
	I_\eps:=[m^*-\eps,m^*+\eps].$
    Fix $c \in [0,\infty)$ and assume that $\alpha_N\sqrt{\frac{N}{b_N}} \longrightarrow c$. On the event $\{ m_1,m_3 \in I_\eps\}$, for $m_2 \in [0,1]$ the inter-block interaction term can be bounded from above by:
\begin{align*}
     \le &\alpha_N \sqrt{b_N\frac{N-b_N}{2}}(m_1+m_3)m_2 
     \le  \alpha_N \sqrt{b_N\frac{N-b_N}{2}} 2(m^*+\eps)m_2 \\ \le& b_N\sqrt{2}c (m^*+\eps) m_2 + b_N \left|\sqrt{2}c- 2\alpha_N \sqrt{\frac{N-b_N}{2b_N}} \right|(m^*+\eps) \\
     =&  b_N\sqrt{2}c m^* m_2 + b_N \left( \sqrt{2}c\eps +o(1)\right) .
\end{align*}
Similarly, we get the lower bound
\begin{multline*}
      \alpha_N \sqrt{b_N\frac{N-b_N}{2}}(m_1+m_3)m_2 \ge  \alpha_N \sqrt{b_N\frac{N-b_N}{2}}2(m^*-\eps)m_2  \\
      \ge  b_N\sqrt{2}c m^* m_2 -b_N \left( \sqrt{2}c\eps +o(1)\right) 
\end{multline*}
 Consider the conditional distribution of $m_2$ under $\mu_{N,\beta,\alpha_N}\left( \, \cdot \, | \, m_1,m_3 \in I_\eps \right)$. By the above bounds, for the probability that $m_2$ is bounded away from $m(c)$ we get
\begin{align}
    &\mu_N( m_2\in [m(c)-\delta, m(c) + \delta]^c\cap [0,1] \, | \, m_1,m_3 \in I_\eps) \nonumber  \\
    =& \frac{\mu_N( m_2 \in [m(c)-\delta, m(c) + \delta]^c\cap [0,1] \, , \,  m_1,m_3 \in I_\eps)}{\mu_N( m_1,m_3 \in I_\eps)}  \nonumber \\
    \le& \exp\left\{ 2 b_N(\sqrt{2}c \eps + o(1)) \right\}\frac{\sum_{m_2 \in [m(c)-\delta, m(c) + \delta]^c\cap [0,1] } \exp\left\{ b_N \left(\frac{\beta}{2} m_2^2 + \sqrt{2}cm^* m_2 \right)\right\} \binom{b_N}{b_N\frac{1+m_2}{2}} }{ \sum_{m_2} \exp\left\{ b_N \left(\frac{\beta }{2}m_2^2 + \sqrt{2}cm^* m_2 \right)\right\}\binom{b_N}{b_N\frac{1+m_2}{2}}} \nonumber \\
    =&\exp\left\{  2 b_N(\sqrt{2}c \eps + o(1)) \right\} \widetilde{\mu}_{b_N,\beta, \sqrt{2}cm^*}\left(m_2 \in [m(c)-\delta, m(c) + \delta]^c\cap [0,1]\right) , \label{eq:magneticField}
\end{align}
where we denote by $\widetilde{\mu}_{b_N,\beta, \sqrt{2}cm^*}$ the Gibbs measure of the Curie-Weiss model of size $b_N$ with external magnetic field $ \sqrt{2}cm^*$. For the Curie--Weiss model with external field $ \sqrt{2}cm^*\ge0$, it is standard that the
	magnetization concentrates exponentially fast around $m(c)$\cite[Chapter 2, p.68, Exercise 2.5]{Velenik_book}. Note that for $c=0$, we have $m(0)=m^*$. Then 
    \begin{align*}
        &\widetilde{\mu}_{b_N,\beta, 0}\left(m_2 \in [m^*-\delta, m^* + \delta]^c\cap [0,1]\right)\\
        =& \widetilde{\mu}_{b_N,\beta, 0}\left(m_2 \in [m^*-\delta, m^* + \delta]^c \, | \, m_2\in  [0,1]\right) \widetilde{\mu}_{b_N,\beta, 0}\left(m_2 \in [0,1]\right) \\
        =& \left(\frac{1}{2} -o(1) \right) \widetilde{\mu}_{b_N,\beta, 0}\left(m_2 \in [m^*-\delta, m^* + \delta]^c \, | \, m_2\in  [0,1]\right)
    \end{align*}
    and conditioned on the set $[0,1]$, the exponential concentration remains true.
	 Hence, for every $c\in [0,\infty)$, there exists a constant $K(\beta,c,\delta)>0$ such that 
\begin{equation}\label{eq:magneticField2}
    \widetilde{\mu}_{b_N,\beta, \sqrt{2}cm^*}\left(m_2 \in [m(c)-\delta, m(c) + \delta]^c\cap [0,1]\right) \le \exp \left\{- K(\beta,c, \delta)b_N \right\}.
\end{equation}
Plugging \eqref{eq:magneticField2} into \eqref{eq:magneticField} yields
\begin{align*}
    &\mu_N( m_2\in [m(c)-\delta, m(c) + \delta]^c\cap [0,1] \, | \, m_1,m_3 \in [m^*-\eps, m^*+\eps]) \nonumber  \\
    \le& \exp\left\{ -b_N\left( K(\beta,c, \delta)-  2(\sqrt{2}c \eps + o(1)) \right) \right\} \longrightarrow 0,
\end{align*}
where the convergence follows for $\eps_{\delta,c,\beta} < \frac{K(\beta,c, \delta)}{2\sqrt{2}c}$ if $c>0$, and any $\eps_{\delta,0,\beta}< \min\{m^*,1-m^*\}$ if $c=0$ respectively, since then for all $\eps < \eps_{\delta,c,\beta}$ and $N$ large enough, $K(\beta,c, \delta)-  2(\sqrt{2}c \eps + o(1)) >0$.
\end{proof}

\begin{lemma}\label{lem:samesign}
	If $\lim_{N\to\infty}\alpha_N\sqrt{b_NN}=\infty$ and
	$\lim_{N\to\infty}\alpha_N\sqrt{\frac{N}{b_N}}=c\in[0,\infty]$,
	then for any $0<\delta<m(c)$ and any
	$0<\eps<\min\{m^*,1-m^*\}$,
	\[
	\frac{\mu_{N,\beta,\alpha_N}\Big(m_1,m_3\in[m^*-\eps,m^*+\eps],\, m_2\le 0\Big)}
	{\mu_{N,\beta,\alpha_N}\Big(m_1,m_3\in[m^*-\eps,m^*+\eps],\, m_2\in[m(c)-\delta,m(c)+\delta]\Big)}
	\longrightarrow 0.
	\]
	Equivalently,
	\[
	\frac{\mu_{N,\beta,\alpha_N}\Big(m_1,m_3\in[m^*-\eps,m^*+\eps],\, m_2\in[m(c)-\delta,m(c)+\delta]\Big)}
	{\mu_{N,\beta,\alpha_N}\Big(m_1,m_3\in[m^*-\eps,m^*+\eps],\, m_2\le 0\Big)}
	\longrightarrow \infty.
	\]
\end{lemma}

\begin{proof}
	Let
	\[
	E_\eps:=\{m_1,m_3\in[m^*-\eps,m^*+\eps]\}.
	\]
	For $m_1,m_3\in[m^*-\eps,m^*+\eps]$ and $m_2\ge 0$ we have
	\[
	\alpha_N\sqrt{b_N\frac{N-b_N}{2}}(m_1+m_3)m_2
	\ge
	2\alpha_N\sqrt{b_N\frac{N-b_N}{2}}(m^*-\eps)m_2.
	\]
	Hence, for $m_2\in[m(c)-\delta,m(c)+\delta]$,
	\[
	\alpha_N\sqrt{b_N\frac{N-b_N}{2}}(m_1+m_3)m_2
	\ge
	2\alpha_N\sqrt{b_N\frac{N-b_N}{2}}(m^*-\eps)(m(c)-\delta).
	\]
	On the other hand, for $m_2\le 0$, the interaction term is at most $0$.
	
	Therefore,
	\begin{align*}
		&\frac{\mu_{N,\beta,\alpha_N}\Big(E_\eps,\, m_2\in[m(c)-\delta,m(c)+\delta]\Big)}
		{\mu_{N,\beta,\alpha_N}\Big(E_\eps,\, m_2\le 0\Big)}\\
		&\qquad\ge
		\exp\left\{
		2\alpha_N\sqrt{b_N\frac{N-b_N}{2}}(m^*-\eps)(m(c)-\delta)
		\right\}
		\frac{\mu_{b_N,\beta}\big(m_2\in[m(c)-\delta,m(c)+\delta]\big)}
		{\mu_{b_N,\beta}(m_2\le 0)}.
	\end{align*}
	By the definition of $m(c)$ and the convergence result for the Curie--Weiss
	model with the corresponding effective external field, we have
	\[
	\mu_{b_N,\beta}\big(m_2\in[m(c)-\delta,m(c)+\delta]\big)\to \frac12,
	\]
	while $\mu_{b_N,\beta}(m_2\le 0)\le \frac12+o(1)$. Thus the second factor
	is bounded away from $0$. Since
	\[
	\alpha_N\sqrt{b_NN}\to\infty,
	\]
	and $(m^*-\eps)(m(c)-\delta)>0$, the exponential term diverges to $\infty$.
	Hence
	\[
	\frac{\mu_{N,\beta,\alpha_N}\Big(E_\eps,\, m_2\in[m(c)-\delta,m(c)+\delta]\Big)}
	{\mu_{N,\beta,\alpha_N}\Big(E_\eps,\, m_2\le 0\Big)}
	\longrightarrow\infty,
	\]
	which is equivalent to the claim.
\end{proof}

In particular, Lemma \ref{lemma:c<infty} and \ref{lem:samesign} together imply that in the case $\lim_{N\to \infty} \alpha_N \sqrt{b_N N} = \infty$ and $\lim_{N\to \infty}\alpha_N\sqrt{\frac{N}{b_N}}=0 $, the block magnetization vector $(m_1,m_2,m_3)$ converges in distribution (with respect to the Gibbs measure) to $\pm (m^*,m^*,m^*)$. Here the first assumption implies that all magnetizations have the same sign and the second one determines the value of the second block (and $m^*$ corresponds to no external magnetic field). For the remainder of this section, we will keep the second assumption but drop the first one. First, we will show, that $m_2$ will still converge to $\pm m^*$:

\begin{lemma}\label{lemma:m2regime2}
    If $\lim_{N \to \infty} \alpha_N\sqrt{\frac{N}{b_N}}=0$, then for all $0< \eps,\delta < \min\{m^*, 1-m^*\}$,
    \begin{multline*}
         \mu_{N,\beta,\alpha_N}\left( (m_1,m_2,m_3) \notin A_\eps\times A_\delta \times A_\eps \right) \\
         \le 4\exp\left\{-n_N\left(c_\beta \eps^2 + o(1) \right) \right\} + \exp\left\{-b_N\left(c_\beta \delta^2 + o(1) \right) \right\} \longrightarrow 0,  
    \end{multline*}
    as $N\to \infty$. In particular it follows that, if we choose sequences $(\eps_N)_N\downarrow0$ and $(\delta_N)_N \downarrow 0$ that satisfy $\delta_N^{-1} = o(\sqrt{b_N})$ and $\eps_N^{-1} = o( \sqrt{N})$, then
    $$
    \mu_{N,\beta,\alpha_N}\left( (m_1,m_2,m_3) \notin A_{\eps_N}\times A_{\delta_N} \times A_{\eps_N} \right) \longrightarrow 0.
    $$
    as $N \to \infty$.
\end{lemma}

\begin{proof}
Again, we will use a similar strategy as in the proof of Lemma \ref{lemma:m1}. More precisely, we bound the interaction term between the blocks, then the sums over $m_1$ and $m_3$ cancel in the numerator and denominator and we get
    \begin{align*}
        &\mu_{N,\beta,\alpha_N}\left( m_2 \notin A_\delta \, | \, m_1,m_3 \in A_\eps \right)
         =\frac{ \mu_{N,\beta,\alpha_N}\left( m_2 \notin A_\delta \, , \, m_1,m_3 \in A_\eps \right)  }{ \mu_{N,\beta,\alpha_N}\left(  m_1,m_3 \in A_\eps \right)  } \\
        \le & \exp\left\{4 \alpha_N \sqrt{b_N \frac{N-b_N}{2}} \right\}\sum_{m_2} \exp\left\{ b_N\left( F_\beta (m_2) - F_{\beta}(\widetilde{m}_{b_N}\right) + R_N^{(2)}(\widetilde{m}_{b_N} ) \right\} \\
        \le& \exp\left\{ - b_N (c_\beta\delta^2 + o(1) )\right\},
    \end{align*}
    where the last step follows since $R_N^{(2)}(\widetilde{m}_{b_N} )  = o(b_N)$ and
    $
    \alpha_N \sqrt{\frac{N-b_N}{2b_N}} 
     \longrightarrow 0 $ if and only if
     $\alpha_N \sqrt{b_N \frac{N-b_N}{2}} = o(b_N).
    $
    Then, 
    \begin{align*}
        &\mu_{N,\beta,\alpha_N} \left( (m_1,m_2,m_3) \notin A_\eps\times A_\delta \times A_\eps \right)  \le 
        2\mu_{N,\beta,\alpha_N}( m_1 \notin A_\eps ) + \mu_{N,\beta,\alpha_N}( m_2 \notin A_\delta) \\
        \le& 2\mu_{N,\beta,\alpha_N}(m_1\notin A_\eps)
    		+\mu_{N,\beta,\alpha_N}(m_1\notin A_\eps \text{ or } m_3\notin A_\eps)  + \mu_{N,\beta,\alpha_N}(m_2\notin A_\delta,\, m_1,m_3\in A_\eps) \\
    		\le& 4\mu_{N,\beta,\alpha_N}(m_1\notin A_\eps)
    		+ \mu_{N,\beta,\alpha_N}(m_2\notin A_\delta \mid m_1,m_3\in A_\eps) \\
            \le & 4\exp\left\{-n_N\left(c_\beta \eps^2 + o(1) \right) \right\} + \exp\left\{-b_N\left(c_\beta \delta^2 + o(1) \right) \right\}
         \longrightarrow 0,
    \end{align*}
    where the convergence follows by the above computation together with Lemma \ref{lemma:m1}
\end{proof}

\begin{proof}[Proof of Theorems \ref{thm:Blockspin} and \ref{thm:PhaseTransition}]

	We will start with the case $\lim_{N\to \infty} \alpha_N\sqrt{b_N N} = \infty$: recall the definition of $\eps_{\delta,c, \beta}$ from Lemma~\ref{lemma:c<infty} for $c \in [0,\infty)$. Let us define $\eps_{\delta,\infty, \beta} :=\frac{m^*}{2}\delta$. Fix $c \in [0,\infty]$ and assume that $\lim_{N\to \infty} \alpha_N\sqrt{\frac{N}{b_N}} = c$. Then, for any $0<\delta < m(c)$ and any $\eps< \min\{\frac{(m^*)^2}{1+m^*}, 1-m^*, \eps_{\delta, c ,\beta}, \delta\}$, we have 
	\begin{flalign*}
		&\mu_{N,\beta,\alpha_N}\left( m\notin B_{\eps}((m^*,m(c),m^*))\cup B_\eps((-m^*,-m(c),-m^*)) \right)  \\
		\le& \mu_{N,\beta,\alpha_N}(m_1 \notin A_\eps) +  \mu_{N,\beta,\alpha_N}(m_3 \notin A_\eps) + \mu_{N,\beta,\alpha_N}(m_1,m_3\in A_\eps, \, m_1m_3<0) \\
		& \quad + \mu_{N,\beta,\alpha_N}(m_1,m_3 \in [m^* - \eps, m^* + \eps] ,  m_2 \in [m(c)-\delta , m(c) + \delta]^c\cap[0,1] ) \\
		& \quad \, \, \, + \mu_{N,\beta,\alpha_N}(m_1,m_3, \in [-m^* - \eps, -m^* + \eps] ,  m_2 \in [-m(c) - \delta , -m(c) + \delta]^c\cap [-1,0]) \\
        & \qquad + \mu_{N,\beta,\alpha_N}(m_1,m_3, \in [m^* - \eps, m^* + \eps] ,  m_2 \le 0) \\
        & \qquad \, \, \, + \mu_{N,\beta,\alpha_N}(m_1,m_3, \in [-m^* - \eps, -m^* + \eps] ,  m_2 \ge 0) \\ & \qquad \qquad \longrightarrow 0
	\end{flalign*}
	where the convergence follows by Lemma \ref{lemma:m1}, Lemma \ref{lemma:m1m3Regime1and2}, Lemma \ref{lem:c=infty} if $c=\infty$ and Lemma \ref{lemma:c<infty} if $c \in [0,\infty)$ respectively, and by Lemma \ref{lem:samesign}.
	In other words, the magnetization vector converges to either $(m^*,m(c),m^*)$ or $(-m^*,-m(c),-m^*)$ in probability with respect to the Gibbs measure. Since the Gibbs measure is invariant under a global spin flip, we obtain that the convergence to either of the two limit points happens with probability $\frac{1}{2}$. This proves the statements in Theorem~\ref{thm:Blockspin}\eqref{thm:Blockspin1},\eqref{thm:Blockspin2} and Theorem~\ref{thm:PhaseTransition} \eqref{thm:Phasetrans1}.
	
	\medskip
	
	Now let us show the case $\lim_{N\to \infty} \alpha_N \sqrt{\frac{N}{b_N}} = 0$: fix $C \in [0,\infty]$, then by Lemma \ref{lemma:m2regime2}, we know that there exist two sequences $(\eps_N), (\delta_N)\downarrow 0$ such that 
$$
\mu_{N,\beta,\alpha_N}\left((m_1,m_2,m_3) \in A_{\eps_N} \times A_{\delta_N} \times A_{\eps_N} \right) \longrightarrow 1.
$$ Hence, it only remains to compute the weights $a(\chi_1,\chi_2, \chi_3,C)$, $\chi_1,\chi_2, \chi_3 \in \{+1,-1\}$ of the possible limiting points as a function of $C$. By a standard diagonalization procedure
Note that by exchangeability of blocks 1 and 3, and by the invariance of the Gibbs measure under a global spin flips, we have to investigate representatives the following three cases
\begin{itemize}
    \item blocks 1 and 3 have the same sign, block 2 has the opposite sign (2 possibilities)
    \item blocks 1 and 3 have opposite sign (4 possibilities)
    \item all three blocks have the same sign (2 possibilities).
\end{itemize}
More precisely, we have
\begin{align}
    &\frac{ \mu_{N,\beta,\alpha_N}\left( m_1,m_3 \in [m^*- \eps_N, m^*+\eps_N], m_2 \in [-m^*-\delta_N, -m^*+\delta_N] \right) }{\mu_{N,\beta,\alpha_N}\left( m_1,m_3 \in [m^*- \eps_N, m^*+\eps_N], m_2 \in [m^*-\delta_N, m^*+\delta_N] \right) } \nonumber \\
    \le& \exp \left\{ \alpha_N \sqrt{b_N \frac{N-b_N}{2}}\left( (-m^*+\delta_N)2(m^*-\eps_N) - (m^*-\delta_N)2(m^*-\eps_N) \right) \right\} \nonumber \\
    =& \exp \left\{ \alpha_N \sqrt{b_N \frac{N-b_N}{2}} 4\left( -(m^*)^2 +m^*(\delta_N+\eps_N) -\eps_N\delta_N \right) \right\} \nonumber \\
    &\longrightarrow  \begin{cases}
          \exp\left\{-2\sqrt{2}C(m^*)^2 \right\} & \quad \text{ if }C \in (0,\infty)\\
          0 & \quad \text{ if }C =\infty \\
          1 & \quad \text{ if }C =0
    \end{cases}, \label{eq:case1upper}
\end{align}
where we used that $\alpha_N \sqrt{b_N \frac{N-b_N}{2}} \longrightarrow \frac{C}{\sqrt{2}}$ as $N \to \infty$. Similarly, one gets the lower bound
\begin{align}
    &\frac{ \mu_{N,\beta,\alpha_N}\left( m_1,m_3 \in [m^*- \eps_N, m^*+\eps_N], m_2 \in [-m^*-\delta_N, -m^*+\delta_N] \right) }{\mu_{N,\beta,\alpha_N}\left( m_1,m_3 \in [m^*- \eps_N, m^*+\eps_N], m_2 \in [m^*-\delta_N, m^*+\delta_N] \right) } \nonumber \\
    \ge& \exp \left\{ \alpha_N \sqrt{b_N \frac{N-b_N}{2}}\left( (-m^*-\delta_N)2(m^*+\eps_N) - (m^*+\delta_N)2(m^*+\eps_N) \right) \right\} \nonumber \\
    =& \exp \left\{ \alpha_N \sqrt{b_N \frac{N-b_N}{2}} 4\left( -(m^*)^2 -m^*(\delta_N+\eps_N) -\eps_N\delta_N \right) \right\} \nonumber\\
    &\longrightarrow  \begin{cases}
          \exp\left\{-2\sqrt{2}C(m^*)^2 \right\} & \quad \text{ if }C \in (0,\infty)\\
          0 & \quad \text{ if }C =\infty \\
          1 & \quad \text{ if }C =0
    \end{cases}, \label{eq:case1lower}
\end{align}
 Putting together \eqref{eq:case1upper} and \eqref{eq:case1lower} yields that 
 \begin{equation}\label{eq:case1}
     a( 1,-1,1,C ) =  \begin{cases}
         \exp\left\{-2\sqrt{2}C(m^*)^2 \right\} a(1,1,1,C) &  \quad \text{ if }C \in (0,\infty) \\
         0 & \quad \text{ if }C =\infty \\
          a(1,1,1,C)  &  \quad \text{ if }C =0
     \end{cases}.
 \end{equation}
 Similarly, one gets the upper and lower bounds
 \begin{align}
    &\frac{ \mu_{N,\beta,\alpha_N}\left( m_1 \in [m^*- \eps_N, m^*+\eps_N], m_2 \in [-m^*-\delta_N, -m^*+\delta_N], m_3 \in [-m^* - \eps_N, -m^* + \eps_N] \right) }{\mu_{N,\beta,\alpha_N}\left( m_1,m_3 \in [m^*- \eps_N, m^*+\eps_N], m_2 \in [m^*-\delta_N, m^*+\delta_N] \right) } \nonumber \\
    \le& \exp \left\{ -\alpha_N \sqrt{b_N \frac{N-b_N}{2}} (m^*- \delta_N)2(m^*-\eps_N)  \right\} \nonumber \\
    =& \exp \left\{ -\alpha_N \sqrt{b_N \frac{N-b_N}{2}} 2\left( (m^*)^2 -m^*(\delta_N+\eps_N) +\eps_N\delta_N \right) \right\} \nonumber \\
    &\longrightarrow \begin{cases}
          \exp\left\{-\sqrt{2}C(m^*)^2 \right\} & \quad \text{ if }C \in (0,\infty)\\
          0 & \quad \text{ if }C =\infty \\
          1 & \quad \text{ if }C =0
    \end{cases}, \label{eq:case2upper}
\end{align}
and 
\begin{align}
    &\frac{ \mu_{N,\beta,\alpha_N}\left( m_1 \in [m^*- \eps_N, m^*+\eps_N], m_2 \in [-m^*-\delta_N, -m^*+\delta_N], m_3 \in [-m^* - \eps_N, -m^* + \eps_N] \right) }{\mu_{N,\beta,\alpha_N}\left( m_1,m_3 \in [m^*- \eps_N, m^*+\eps_N], m_2 \in [m^*-\delta_N, m^*+\delta_N] \right) } \nonumber \\
    \ge& \exp \left\{ -\alpha_N \sqrt{b_N \frac{N-b_N}{2}} (m^*- \delta_N)2(m^*-\eps_N)  \right\} \nonumber \\
    =& \exp \left\{ -\alpha_N \sqrt{b_N \frac{N-b_N}{2}} 2\left( (m^*)^2 +m^*(\delta_N+\eps_N) +\eps_N\delta_N \right) \right\} \nonumber \\
    &\longrightarrow \begin{cases}
          \exp\left\{-\sqrt{2}C(m^*)^2 \right\} & \quad \text{ if }C \in (0,\infty)\\
          0 & \quad \text{ if }C =\infty \\
          1 & \quad \text{ if }C =0
    \end{cases}. \label{eq:case2lower}
\end{align}
Again, \eqref{eq:case2upper} and \eqref{eq:case2lower} together yield that
\begin{align}
    a( 1,1,-1,C)= \begin{cases}
          \exp\left\{-\sqrt{2}C(m^*)^2 \right\}a(1,1,1,C) & \quad \text{ if }C \in (0,\infty)\\
          0 & \quad \text{ if }C =\infty \\
          a(1,1,1,C) & \quad \text{ if }C =0
    \end{cases}. \label{eq:case2}
\end{align}
Lastly, using \eqref{eq:case1}, \eqref{eq:case2} and the fact that
\begin{flalign*}
    a(1,-1,1,C) &= a(-1,1,-1,C) \\
    \text{and} \quad  a(1,1,-1,C)&= a(-1,-1,1,C) = a(1,-1,-1,C) = a(-1,1,1,C),
\end{flalign*}
by the normalization of probability measures for all $C \in [0,\infty]$, we obtain the condition
$$
a(1,1,1,C)(2+ 2e^{-2\sqrt{2}C(m^*)^2}+  4 e^{-\sqrt{2}C(m^*)^2} ) =1,
$$
yielding
$$
a(1,1,1,C) = \frac{1}{2\left( 1+ e^{-\sqrt{2}C(m^*)^2} \right)^2}.
$$
This proves the statements in Theorem~\ref{thm:Blockspin}\eqref{thm:Blockspin2},\eqref{thm:Blockspin3} and in Theorem~\ref{thm:PhaseTransition}\eqref{thm:Phasetrans2}.

\end{proof}

\textbf{Acknowledgements}

Funded by the Deutsche Forschungsgemeinschaft (DFG, German Research Foundation) under Germany's Excellence Strategy EXC 2044/2 -390685587, Mathematics M\"unster: Dynamics-Geometry-Structure.

\bibliographystyle{abbrv}

\end{document}